\documentclass[11pt]{article}
\usepackage[left=3.5cm, right=3cm, top=2cm]{geometry}
\usepackage{amsmath,amssymb,amsthm,graphicx,subfigure,float,url} \usepackage[colorlinks,linkcolor=red,anchorcolor=green,citecolor=blue]{hyperref}
\usepackage{amsfonts,amssymb}
\usepackage{amsmath,mathtools}
\usepackage{graphicx}
\usepackage{cite}
\usepackage{enumerate}
\def\R{\textrm{I\kern-0.21emR}}
\def\N{\textrm{I\kern-0.21emN}}

\renewcommand{\geq}{\geqslant}
\renewcommand{\leq}{\leqslant}

\renewcommand{\geq}{\geqslant}
\renewcommand{\leq}{\leqslant}

\newcommand {\ddx} {{\rm d}x}
\newcommand {\dd} {{\rm d}}

\DeclareMathOperator\supp{supp}

\newcommand {\beq} {\begin{equation}}
\newcommand {\eeq} {\end{equation}}

\newcommand {\e}  {\varepsilon}

\newcommand {\Chi} {{\bf \raise 2pt \hbox{$\chi$}} }

\newcommand {\sgn} { {\rm sgn} }

\newtheorem{theorem}{Theorem}  
\newtheorem{proposition}{Proposition}

\newtheorem{lemma}{Lemma}

\theoremstyle{definition}\newtheorem{remark}{Remark}

   \allowdisplaybreaks
\begin{document}
 \title{Asymptotic analysis of selection-mutation models in the presence of multiple fitness peaks 
}

	\author{Tommaso Lorenzi\thanks{School of Mathematics and Statistics, University of St Andrews, Scotland (tl47@st-andrews.ac.uk).} \and
	Camille Pouchol\thanks{Department of Mathematics, KTH Royal institute of Technology, SE-100 44 Stockholm, Sweden (pouchol@kth.se)} \thanks{Sorbonne Universit\'e, CNRS UMR 7598, Laboratoire Jacques-Louis Lions, F-75005, Paris, France}}

\date{}

         \pagestyle{myheadings} \markboth{Asymptotic analysis of phenotype-structured models in the presence of multiple fitness peaks}{T. Lorenzi and C. Pouchol} \maketitle

          \begin{abstract}
We study the long-time behaviour of phenotype-structured models describing the evolutionary dynamics of asexual populations whose phenotypic fitness landscape is characterised by multiple peaks. First we consider the case where phenotypic changes do not occur, and then we include the effect of heritable phenotypic changes. In the former case the model is formulated as an integrodifferential equation for the phenotype distribution of the individuals in the population, whereas in the latter case the evolution of the phenotype distribution is governed by a non-local parabolic equation whereby a linear diffusion operator captures the presence of phenotypic changes. We prove that the long-time limit of the solution to the integrodifferential equation is unique and given by a measure consisting of a weighted sum of Dirac masses centred at the peaks of the phenotypic fitness landscape. We also derive an explicit formula to compute the weights in front of the Dirac masses. Moreover, we demonstrate that the long-time solution of the non-local parabolic equation exhibits a qualitatively similar behaviour in the asymptotic regime where the diffusion coefficient modelling the rate of phenotypic change tends to zero. However, we show that the limit measure of the non-local parabolic equation may consist of less Dirac masses, and we provide a sufficient criterion to identify the positions of their centres. Finally, we carry out a detailed characterisation of the speed of convergence of the integral of the solution (\emph{i.e.} the population size) to its long-time limit for both models. Taken together, our results support a more in-depth theoretical understanding of the conditions leading to the emergence of stable phenotypic polymorphism in asexual populations. 
          \end{abstract}


\section{Introduction}
\label{Introduction}
Phenotype-structured models formulated as integrodifferential equations (IDEs) or non-local partial differential equations (PDEs) have been increasingly used as a theoretical framework to study evolutionary dynamics in a variety of asexual populations~\cite{alfaro2017effect,almeida2019evolution,bootsma2012modeling,bouin2014travelling,bouin2012invasion,burger1996stationary,busse2016mass,calsina2007asymptotic,calsina2005stationary,calsina2004small,chisholm2016evolutionary,chisholm2015emergence,delitala2013mathematical,domschke2017structured,iglesias2018long,kimura1965stochastic,lorenzi2019structured,lorenzi2016tracking,lorenzi2015dissecting,lorenzi2015mathematical,lorenzi2018role,lorz2015modeling,lorz2013populational,magal2000mutation,Nordmann2018,olivier2017combination,Perthame2006,pouchol2018asymptotic}. In these models, the phenotypic state of each individual is represented by a continuous real variable $x$, and the phenotypic distribution of the individuals within the population at a given time $t \geq 0$ is described by a function $n(t,x)~\geq~0$. In many scenarios of biological and ecological interest one can assume $x \in \Omega$, where $\Omega$ is a smooth bounded domain of $\mathbb{R}^{d}$, $d \geq 1$. 

We focus here on the case where, in the absence of phenotypic changes, the evolution of the population density function $n(t,x)$ is governed by an IDE of the form
\beq
\label{IDE}
\frac{\partial n}{\partial t}  =  R(x,\rho(t)) \, n, \quad n \equiv n(t,x), \quad (t,x) \in (0,\infty) \times \overline{\Omega}.
\eeq
The function $R(x,\rho(t))$ represents the net per capita growth rate of individuals in the phenotypic state $x$, under the environmental conditions determined by the population size 
\beq
\label{rhoIDE}
\rho(t) := \int_\Omega n(t,x) \, \ddx,
\eeq
and models the effect of natural selection. In fact, depending on the sign of $R(x,\rho(t))$ the number density of individuals in the phenotypic state $x$ will either grow or decay at time~$t$. The function $R(x,\rho(t))$ can thus be seen as the phenotypic fitness landscape of the population~\cite{fragata2018evolution,poelwijk2007empirical}.  

In the ecological and biological scenarios whereby the effect of heritable, spontaneous phenotypic changes need to be taken into account, a linear diffusion operator can be included in the IDE~\eqref{IDE}. This leads to a non-local parabolic PDE of the form
\beq
\label{PDE}
\displaystyle{\frac{\partial n}{\partial t} = R(x,\rho(t)) \, n + \beta \Delta n, \quad n \equiv n(t,x), \quad (t,x) \in (0,\infty) \times \Omega},
\eeq
where the diffusion coefficient $\beta > 0$ models the rate of phenotypic change. Since phenotypic changes preserve the total number of individuals within the population, zero Neumann is the most natural choice of boundary conditions for the non-local PDE~\eqref{PDE}. 

The way in which the fitness function $R(x,\rho(t))$ is defined depends largely on the underlying application problem, and we consider here the prototypical definition
\beq
\label{defR}
R(x,\rho(t)) := r(x)-\rho(t).
\eeq
In definition~\eqref{defR}, the function $r(x)$ is the net per capita growth rate of the individuals in the phenotypic state $x$ (\emph{i.e.} the difference between the rate of proliferation through asexual reproduction and the rate of death under natural selection). Hence the maximum points of this function correspond to the fitness peaks (\emph{i.e.} the peaks of the phenotypic fitness landscape of the population). Moreover, the saturating term $-\rho(t)$ models the limitations on population growth imposed by carrying capacity constraints (\emph{e.g.} limited availability of space and resources). 

In the framework of the IDE~\eqref{IDE}, or the non-local PDE~\eqref{PDE}, a mathematical depiction of phenotypic adaptation can be obtained by studying the long-time behaviour of the population density function $n(t,x)$. In this regard, whilst the case of one single fitness peak has been broadly studied~\cite{ackleh2005rate,barles2009concentration,calsina2013asymptotics,chisholm2016effects,diekmann2005dynamics,lorz2011dirac,perthame2008dirac}, there is a paucity of literature concerning the case where multiple fitness peaks are present, with the exception of the asymptotic results presented in~\cite{alfaro2018evolutionary,busse2016mass,Coville2013,desvillettes2008selection,djidjou2017steady,Jabin2011,lorenzi2019structured,lorenzi2014asymptotic,Perthame2006,Pouchol2017b}. Based on these few previous results, we expect the solution to the IDE~\eqref{IDE} complemented with~\eqref{defR} to converge to a limit measure given by a sum of weighted Dirac masses centred at the maximum points of the function $r(x)$ as $t \to \infty$, and we envisage the long-time solution of the non-local PDE~\eqref{PDE} to exhibit a qualitatively similar behaviour in the asymptotic regime where $\beta \to 0$. This represents a mathematical formalisation of the biological notion that the phenotypic variants corresponding to the fitness peaks (\emph{i.e.} the fittest phenotypic variants) are ultimately selected. However, the following key questions still remain open: is there a unique weighted sum of Dirac masses defining the limit of the solution to the IDE~\eqref{IDE} when $t \to \infty$? If so, what are the only admissible values of the weights in front of the different Dirac masses? Are there any differences between the long-time solution to the IDE~\eqref{IDE} and the asymptotic limit when $\beta \to 0$ of the long-time solution to the non-local PDE~\eqref{PDE}? If so, how does the presence of the diffusion term change the values of the weights associated with the Dirac masses? How does the speed of convergence of the population size $\rho(t)$ to its long-time limit differs between the IDE~\eqref{IDE} and the non-local PDE~\eqref{PDE}?

In this paper, we address these questions focussing on the case where the fitness function is defined via~\eqref{defR}. In summary, using Laplace's method we derive an explicit formula to compute the weights in front of the Dirac masses that constitute the long-time limit of the solution to the IDE~\eqref{IDE}, and show that the weights are uniquely determined by the initial condition $n(0,x)$ and the Hessian $H(r)$ at the maximum points of the function $r(x)$ (\emph{vid.} Theorem~\ref{Weight_IDE}). Moreover, exploiting the properties of the principal eigenpair of the elliptic differential operator $\beta \Delta + r$, we prove that when $\beta \to 0$ the long-time limit of the non-local PDE~\eqref{PDE} converges to a limit measure consisting of a weighted sum of Dirac masses with non-negative weights centred at the maximum points of the function $r(x)$ (\emph{vid.} Proposition~\ref{Agreement}). We also derive sufficient conditions for the limit measure to be unique and demonstrate that, {\it ceteris paribus}, there can be a significant difference between this limit measure and the limit measure of the IDE~\eqref{IDE}. In particular, we show that, unless {\it ad hoc} symmetry assumptions are made (\emph{vid.} Proposition~\ref{propsym}), the limit measure for the non-local PDE~\eqref{PDE} may consist of less Dirac masses than the limit measure for the IDE~\eqref{IDE} (\emph{i.e.} a smaller number of weights will be strictly positive), and we provide a sufficient criterion to identify the maximum points of the function $r(x)$ corresponding to the Dirac masses with positive weights (\emph{vid.} Proposition~\ref{propnosym}). This criterion relies on a suitable multidimensional characterisation of the concavity of the function $r(x)$ at the maximum points which is borrowed from semiclassical analysis. Finally, we carry out a detailed characterisation of the speed of convergence of the population size $\rho(t)$ to its long-time limit both for the IDE~\eqref{IDE} and for the non-local PDE~\eqref{PDE}. Taken together, our results support a more in-depth theoretical understanding of the conditions leading to the emergence of stable polymorphism in asexual populations. 

The remainder of the paper is organised as follows. In Section~\ref{Problem} we introduce our main assumptions and a few technical preliminaries. In Section~\ref{secIDE} we carry out a qualitative and quantitative characterisation of the solution to the IDE~\eqref{IDE} complemented with~\eqref{defR} when $t \to \infty$, while in Section~\ref{secPDE} we study the asymptotic properties of the solution to the non-local PDE~\eqref{PDE} complemented with~\eqref{defR} by letting first $t \to \infty$ and then $\beta \to 0$. In both sections, we present a sample of numerical solutions that confirm the analytical results obtained. Section~\ref{secresper} concludes the paper by providing a brief overview of possible research perspectives.

\section{Main assumptions, notation and preliminaries}
\label{Problem}
In this paper,  we will consider the following Cauchy problem for the IDE~\eqref{IDE}
\begin{equation}
\begin{cases}
\label{Basis}
\displaystyle{\frac{\partial n}{\partial t}  =  \left(r(x)-\rho(t)\right) n, \quad n \equiv n(t,x), \quad (t,x) \in (0,\infty) \times \overline{\Omega}},
\\\\
n(0,x) = n^0(x), \quad n^0 \in C(\overline{\Omega}), \quad n^0 \geq 0, \quad n^0 \not\equiv 0
\end{cases}
\end{equation}
and the following initial-boundary value problem for the non-local parabolic PDE~\eqref{PDE} 
\begin{equation}
\begin{cases}
\label{BasisMut}
\displaystyle{\frac{\partial n_{\beta}}{\partial t} - \beta \Delta n_{\beta} =  \left(r(x)-\rho_{\beta}(t)\right) n_{\beta}, \quad n_{\beta} \equiv n_{\beta}(t,x), \quad (t,x) \in (0,\infty) \times \Omega},
\\\\
\nabla n_{\beta}(t,x) \cdot \nu(x) = 0, \quad (t,x) \in (0,\infty) \times \partial \Omega,
\\\\
n_{\beta}(0,x) = n^0(x), \quad  n^0 \in C(\Omega), \quad n^0 \geq 0, \quad n^0 \not\equiv 0, 
\end{cases}
\end{equation}
where $\nu(x)$ is the outward normal to the boundary $\partial \Omega$ at the point $x \in \partial \Omega$. 
\\
\paragraph{Main assumptions on the function $r(x)$.} 
In order to prevent $n(t,\cdot)$ from vanishing as $t \to \infty$, we will assume 
\begin{equation}
\label{rM}
r : \overline{\Omega} \to \mathbb{R}, \quad r \in C(\overline{\Omega}), \quad  \max_{x \in \Omega}  \, r(x) = r_M > 0
\end{equation}
and, being interested in the case where $r(x)$ has multiple maximum points, we will also assume
\begin{equation}
\label{Regularity2}
\arg \max(r) = \{\bar{x}_1, \ldots, \bar{x}_N\} \subset \Omega \quad \text{with} \quad N \geq 2.
\end{equation}
Notice that we assume all points $\bar{x}_i$ to belong to the interior of $\Omega$ in order to simplify the presentation. However, part of our results can be extended to the case where the set $\arg \max(r)$ contains some boundary points and, when appropriate, we will comment on how the proofs presented here could be adapted to such a case. Where necessary, we will make the additional assumptions 
\beq
\label{notemptyinters}
\supp(n^0) \, \cap \, \arg \max(r) \neq \emptyset
\eeq 
and
\begin{equation}
\label{Non_Degenerate}
r \in C^{2}(\overline{\Omega}), \quad  \det(H_i) < 0  \quad \text{for} \quad i =1, \ldots, N,
\end{equation} 
where $H_i$ is the Hessian $H(r)$ evaluated at the point $\bar{x}_i$. Assumptions~\eqref{Non_Degenerate} ensure that each maximum point $\bar{x}_i$ is nondegenerate. 
\\

\paragraph{Laplace's method.} We recall some useful results on the asymptotic expansion of integrals involving exponentials, usually referred to as Laplace's method. We refer the reader to~\cite{Wong2001} for a general presentation of such an asymptotic method and for the related proofs. Let $r$ satisfy assumptions~\eqref{rM}, \eqref{Regularity2} and \eqref{Non_Degenerate}, and assume $f \in C(\Omega)$. For $\bar{x}_i \in \arg \max(r)$ and $\e$ small enough so that $\bar{x}_i$ is the only maximum point of $r(x)$ in the ball $B(\bar{x}_i, \e)$, Laplace's method ensures that if $\bar{x}_i \in \supp\left(f \right)$ then
\beq
\label{Laplace}
\int_{B(\bar{x}_i, \e)} f(x) e^{r(x) t} \, \ddx \sim (2 \pi)^{d/2} \frac{f(\bar{x}_i)}{\sqrt{|\det(H_i)|}} \frac{e^{r_M t}}{t^{\frac{d}{2}}}\quad \text{as } t \to \infty, 
\eeq
whereas if $\bar{x}_i \notin \supp\left(f \right)$ then
\beq
\label{Laplace_zero}
\int_{B(\bar{x}_i, \e)} f(x) e^{r(x) t} \, \ddx  = o \left(\frac{e^{r_M t}}{t^{\frac{d}{2}}}\right) \quad \text{as } t \to \infty.
\eeq
Moreover, if $f(x)$ and $r(x)$ are, respectively, of class $C^1$ and $C^3$ on the ball $B(\bar{x}_i, \e)$, higher order terms of the asymptotic expansion can be computed. In this case, Laplace's method ensures that
\beq
\label{Laplace_further}
\int_{B(\bar{x}_i, \e)}  f(x) e^{r(x) t} \, \ddx  = \left(A_i + \frac{B_i}{t} + o \left( \frac{1}{t}  \right)  \right) \frac{e^{r_M t}}{t^{\frac{d}{2}}}\quad \text{as } t \to \infty,
\eeq
where $A_i$ and $B_i$ are real constants, the values of which are not relevant for our purposes.
\smallskip

\paragraph{Preliminaries about the operator $\beta \Delta + r$.}
We will consider the elliptic differential operator 
\beq
\label{Abeta}
L_{\beta} = \beta \Delta + r
\eeq
acting on functions defined on $\Omega$, and we will denote by $(\lambda_{\beta},\psi_{\beta})$ the principal eigenpair of $L_{\beta}$ (\emph{i.e.} $L_{\beta} \psi_{\beta} = - \lambda_{\beta} \psi_{\beta}$) with zero Neumann boundary condition. A useful result is established by the following lemma, which follows from the Krein-Rutman theorem~\cite{Evans2010,Leman2015}.
\\
\begin{lemma}
\label{eigenpair}
The principal eigenvalue $\lambda_\beta$ is simple and there is a unique normalised positive eigenfunction $\psi_\beta$ associated with $\lambda_\beta$. The eigenfunction $\psi_\beta$ is smooth and the principal eigenvalue $\lambda_\beta$ is given by 
\beq
\label{lambdabeta}
\lambda_\beta = \inf_{\phi \in{H^1(\Omega) \setminus \{0\}}} \mathcal{R}(L_{\beta},\phi),
\eeq
where $\mathcal{R}$ denotes the Rayleigh quotient, i.e.
$$
\mathcal{R}(L_{\beta},\phi) := \frac{\beta \displaystyle{\int_\Omega \left|\nabla \phi(x) \right|^2 \, \ddx - \int_\Omega r(x) \phi^2(x) \, \ddx}}{\displaystyle{\int_\Omega \phi^2(x) \, \ddx}}.
$$
The infimum is attained only when $\phi$ is a multiple of $\psi_\beta$.
\end{lemma}

We will also make use of the second eigenvalue $\lambda_{2,\beta} < \lambda_\beta$ of the operator $L_{\beta}$. For a function $u \in L^2(\Omega)$, we will denote by $\alpha_2(u)$ the $L^2(\Omega)$-projection of $u$ onto the finite-dimensional eigenspace associated to $\lambda_{2, \beta}$, and we will denote the opposite to the spectral gap of the operator $L_{\beta}$ by 
\beq
\label{def:gammabeta}
\gamma_\beta := \lambda_{2,\beta} - \lambda_\beta < 0.
\eeq

If the set $\Omega$ is symmetric with respect to some hyperplane $S$, which without loss of generality we will define as 
\beq
\label{defS}
S:=\{x_1=0\},
\eeq
then
\beq
\label{Omegasym}
x=(x_1, x_2, \ldots, x_d) \in \Omega \quad \implies \quad (-x_1, x_2, \ldots, x_d) \in \Omega. 
\eeq
Under assumption~\eqref{Omegasym} and given a function $f : \overline{\Omega} \to \mathbb{R}$, we will use the notation
$$
\hat{f}(x) = f(-x_1, x_2, \ldots, x_d) \quad \text{for } \; x \in \Omega.
$$
With this notation, we will say $f$ to be symmetric with respect to the hyperplane $S$ if $\hat{f}(x)=f(x)$ for all $x \in \Omega$. A useful result is established by the following lemma.
\\

\begin{lemma}
\label{Eigen}
If the set $\Omega$ satisfies assumption~\eqref{Omegasym} and the function $r(x)$ is symmetric with respect to the hyperplane $S$, i.e. if
\beq
\label{rsym}
\hat{r}(x) = r(x) \quad \text{for all } \; x \in \Omega, 
\eeq
then the principal eigenfunction $\psi_{\beta}$ is symmetric with respect to the hyperplane $S$.
\end{lemma}

\begin{proof}
Since $L_{\beta} \psi_\beta = - \lambda_{\beta} \, \psi_{\beta}$, that is,
$$
\beta \Delta \psi_{\beta} + r \, \psi_\beta = -\lambda_{\beta} \, \psi_{\beta} \quad \text{in } \; \Omega,
$$
if assumption~\eqref{Omegasym} is satisfied then
$$
\beta \Delta \hat{\psi}_{\beta} + \hat{r} \, \hat{\psi}_\beta = -\lambda_{\beta} \, \hat{\psi}_{\beta} \quad \text{in } \; \Omega.
$$
Under the additional assumption~\eqref{rsym} the latter elliptic equation implies that
$$
\beta \Delta \hat{\psi}_{\beta} + r \, \hat{\psi}_\beta = -\lambda_{\beta} \, \hat{\psi}_{\beta} \quad \text{in } \; \Omega,
$$
that is, $L_{\beta} \hat{\psi}_{\beta} = - \lambda_{\beta} \, \hat{\psi}_{\beta}$. Hence $\hat{\psi}_{\beta}$ is a normalised positive eigenfunction associated with the principal eigenvalue $\lambda_{\beta}$, and the uniqueness of the principal eigenfunction of the operator $L_{\beta}$ ensures that $\hat{\psi}_{\beta} = \psi_{\beta}$.
\end{proof}

\paragraph{Function space framework}
Throughout the paper, we will consider the space of Radon measures $\mathcal{M}^1\big(\overline{\Omega}\big)$ as the dual of the space of continuous functions $C\big(\overline{\Omega}\big)$. With a slight abuse of notation, the integral of a function $\varphi \in C\big(\overline{\Omega}\big)$ against a measure $\mu \in \mathcal{M}^1\big(\overline{\Omega}\big)$ will be denoted in the same way as the integral of the product of the function $\varphi$ with an $L^1$-function, \emph{i.e.}
$$
\int_{\Omega} \varphi(x) \, \mu(x) \, {\rm d}x \; = \; \int_{\Omega} \varphi \; {\rm d}\mu.  
$$
Sequences of functions $f_k$ in $L^1(\Omega)$ will be regarded as elements of the bigger space $\mathcal{M}^1\big(\overline{\Omega}\big)$. Given a sequence $\mu_k$ in $\mathcal{M}^1\big(\overline{\Omega}\big)$, we will write $\displaystyle{\mu_k  \xrightharpoonup[k  \rightarrow \infty]{} \mu}$ to indicate the weak-$*$ convergence of $\mu_k$ to $\mu$, namely that 
$$
\int_{\Omega} \varphi(x) \, \mu_k(x) \, {\rm d}x \; \xrightarrow[k\rightarrow\infty] \; \int_{\Omega} \varphi(x) \, \mu(x) \, {\rm d}x \quad \forall \, \varphi \in C\big(\overline{\Omega}\big).
$$
Moreover, we will say that the sequence $\mu_k$ concentrates on a set $\omega \subset \Omega$ if 
$$
\int_{\Omega}\varphi(x) \, \mu_k(x) \, {\rm d}x \xrightarrow[k\rightarrow\infty] \, 0 \quad \forall \,  \varphi \in C\big(\overline{\Omega}\big) \; \text{ s.t. } \supp(\varphi) \cap \omega = \emptyset,
$$
and we will use the result given by the following lemma, which is a well-known fact in measure theory.
\\
\begin{lemma}
\label{lemmaconc}
If a sequence $\mu_k$ in $\mathcal{M}^1\big(\overline{\Omega}\big)$ concentrates on a finite set $\{\bar{x}_1, \ldots , \bar{x}_N\}$ and $\displaystyle{\mu_k  \xrightharpoonup[k  \rightarrow \infty]{} \mu}$ then the limit measure $\mu$ must be a linear combination of Dirac masses centred at the points $\bar{x}_1, \ldots, \bar{x}_N$.
\end{lemma}
Finally, we will say that a measure $\mu \in \mathcal{M}^1\big(\overline{\Omega}\big)$ is symmetric with respect to the hyperplane $S$ if 
$$
\int_{\Omega} \varphi(x) \, \mu(x) \, {\rm d}x \; = \; \int_{\Omega} \hat{\varphi}(x) \, \mu(x) \, {\rm d}x \quad \forall \, \varphi \in C\big(\overline{\Omega}\big) \quad \text{s.t.} \quad \hat{\varphi}(x) = \varphi(x) \;\; \forall \, x \in \Omega
$$
and we will use the result given by the following lemma, the proof of which is straightforward.
\\
\begin{lemma}
\label{Symmetry}
Let $\mu_k$ be a sequence of symmetric measures in $\mathcal{M}^1\big(\overline{\Omega}\big)$. If $\displaystyle{\mu_k  \xrightharpoonup[k  \rightarrow \infty]{} \mu}$ then the limit measure $\mu$ is symmetric as well.
\end{lemma}

\section{Long-time behaviour of the Cauchy problem \eqref{Basis}}
\label{secIDE}
In this section, we study the asymptotic behaviour of the solutions to the Cauchy problem~\eqref{Basis} when $t \to \infty$ (Section~\ref{IDEanalysis}), and we provide a sample of numerical solutions that confirm the analytical results obtained (Section~\ref{IDEsimul}). 

\subsection{Asymptotic analysis}
\label{IDEanalysis}
Under assumptions~\eqref{rM} there exists a unique non-negative solution $n \in C([0,+\infty); L^1(\Omega))$ of the Cauchy problem~\eqref{Basis}~\cite{desvillettes2008selection,Perthame2006}. Moreover, solving the Cauchy problem~\eqref{Basis} yields the semi-explicit formula
\beq
\label{Implicit} 
n(t,x) = n^0(x)  \, \displaystyle{e^{r(x) t - \int_{0}^t \rho(s) \, \dd s }}
\eeq
and, if assumption~\eqref{Regularity2} is satisfied as well, the solution $n(t,x)$ is known to concentrate on the set $\arg \max(r)$ when $t \to \infty$, as established by the following theorem.
\\
\begin{theorem}
\label{Theorem1}
Under assumptions~\eqref{rM}-\eqref{notemptyinters}, the solution to the Cauchy problem~\eqref{Basis} is such that
\beq
\label{e.th1rho}
\rho(t) \xrightarrow[t \to \infty]{} r_M
\eeq 
and, up to extraction of subsequences,
\beq
\label{e.th1n}
n(t,x)  \xrightharpoonup[t  \rightarrow \infty]{} r_M \, \sum_{i=1}^N a_i \, \delta_{\bar{x}_i}(x) \;\; \text{ with } \;\; a_i \geq 0 \;\; \text{ and } \;\; \sum_{i=1}^N a_i =1.
\eeq 
\end{theorem}
In the case where 
\beq
\label{rm}
\min_{x \in \Omega}  \, r(x) = r_m > 0,
\eeq
Theorem~\ref{Theorem1} can be proved through a few simple calculations building upon the method presented in~\cite{perthame2008dirac}, as shown by the proof provided in the Appendix for the sake of completeness. Alternatively, when assumption~\eqref{rm} is satisfied, one can prove Theorem~\ref{Theorem1} using the method of proof presented in~\cite{Perthame2006}, which relies on the observation that the semi-explicit solution~\eqref{Implicit} can be made explicit as
$$
\displaystyle{n(t,x) = \frac{\displaystyle{n^0(x) \, e^{r(x)t}}}{1 + \displaystyle{\int_\Omega\frac{ n^0(y)}{r(y)}(e^{r(y)t}-1)\, {\rm d}y}}}.
$$
On the other hand, in the case where assumption~\eqref{rm} is not satisfied, the proof is much more intricate and requires the use of a Lyapunov functional of the form
$$
W(t) = \int_\Omega \big[n(t,x) - n^\infty(x) - n^\infty(x) \ln \left(n(t,x)\right) \big] \, \ddx,
$$
with $n^\infty$ being any measure concentrated on the set $\arg \max(r)$ and having total mass $r_M$. We refer the interested reader to~\cite{Jabin2011, Pouchol2017b} for a proof of Theorem~\ref{Theorem1} in such a more general case. 

We prove here that the coefficients $a_1, \ldots, a_N$ that define the limit measure~\eqref{e.th1n} are uniquely determined by the initial condition and by the Hessian of $r(x)$ at the points $\bar{x}_1, \ldots, \bar{x}_N$, which entails the convergence of the whole trajectory $n(t,\cdot)$ to a unique limit point as $t \to \infty$. These results are summarised by the following theorem, which also provides a characterisation of the rate of convergence of the total mass $\rho(t)$ to the long-term limit $r_M$.
\\
\begin{theorem}
\label{Weight_IDE}
If assumptions \eqref{rM}-\eqref{Non_Degenerate} are satisfied, then the solution to the Cauchy problem~\eqref{Basis} is such that
\beq
\label{e.prop}
n(t,x)  \xrightharpoonup[t  \rightarrow \infty]{} r_M \, \sum_{i=1}^N a_i \, \delta_{\bar{x}_i}(x) 
\eeq 
with
\beq
\label{alphaiIDE}
a_i = A \frac{n^0(\bar{x}_i)}{\sqrt{|\det(H_i)|}} \quad \text{for } \; i=1, \ldots, N,
\eeq
where $A>0$ is a normalising constant such that $\displaystyle{\sum_{i=1}^N a_i = 1}$. 

\noindent
Moreover, if the functions $n^0(x)$ and $r(x)$ are, respectively, of class $C^1$ and $C^3$ in a neighbourhood of each maximum point $\bar{x}_1, \ldots, \bar{x}_N$ then  
\beq
\label{rhosim}
r_M - \rho(t) \sim \frac{d}{2t} \quad \text{as } t \to \infty. 
\eeq
\end{theorem}
\begin{proof}
Throughout the proof we will use the following notation 
\beq
\label{defJ}
J(t) := \frac{e^{r_M t} e^{- \int_0^{t} \rho(s) \, ds}}{t^{d/2}}.
\eeq

The results established by Theorem~\ref{Theorem1} ensure that we can extract a subsequence $n(t_k,\cdot)$ such that
$$
n(t_k,\cdot) \xrightharpoonup[k  \rightarrow \infty]{} r_M \, \sum_{i=1}^N a_i \, \delta_{\bar{x}_i}(x) \;\; \text{ with } \;\; \sum_{i=1}^N a_i =1.
$$
Moreover, considering $\e>0$ small enough so that $\bar{x}_i$ is the only maximum point of the function $r(x)$ in the ball $B(\bar{x}_i, \e)$ for every $i=1, \ldots, N$ and integrating over $B(\bar{x}_i, \e)$ both sides of the expression~\eqref{Implicit} for $n(t,x)$ with $t=t_k$ we find that
\beq
\label{covktoinf}
\int_{B(\bar{x}_i, \e)} n(t_k,x) \, \ddx = \left(\int_{B(\bar{x}_i, \e)} n^0(x) \, e^{r(x) t_k} \, \ddx\right) \, \displaystyle{e^{- \int_0^{t_k} \rho(s) \, \dd s}}. 
\eeq
The long-time behaviour of the integral on the right-hand side of the above equation can be characterised using Laplace's method. In particular, the asymptotic relation~\eqref{Laplace} ensures that if $\bar{x}_i \in \supp(n^0)$ then
\beq
\label{Laplace_est}
\int_{B(\bar{x}_i, \e)} n(t_k,x) \, \ddx \sim (2 \pi)^{d/2} \frac{n^0(\bar{x}_i)}{\sqrt{|\det(H_i)|}}  \; J(t_k) \quad \text{as } k \to \infty,
\eeq
where $J$ is defined according to~\eqref{defJ}. On the other hand, the asymptotic relation~\eqref{Laplace_zero} ensures that if $\bar{x}_i \notin \supp(n^0)$ then
\beq
\label{Laplace_est_zero}
\int_{B(\bar{x}_i, \e)} n(t_k,x) \, \ddx  = o \left(\frac{1}{\sqrt{|\det(H_i)|}}  \; J(t_k)\right) \quad \text{as } k \to \infty.
\eeq
Taken together, the integral identity~\eqref{covktoinf} and the asymptotic relations~\eqref{Laplace_est} and \eqref{Laplace_est_zero} allow us to conclude that there exist some real constants $K>0$ and $A>0$ such that
\beq
\label{limitlap}
J(t_k)  \; \xrightarrow[k\rightarrow\infty] \; K
\eeq
and
$$
a_i = A \, \frac{n^0(\bar{x}_i)}{\sqrt{|\det(H_i)|}} \quad \text{with} \quad \sum_{i=1}^N a_i = 1.
$$
We remark that $K$ cannot be $0$ because otherwise $\rho(t)$ would converge to $0$, which cannot be since $\rho(t) \to r_M$ [\emph{cf.} the asymptotic result~\eqref{e.th1rho}]. Hence, the coefficients $a_i$ that define the limit measure are uniquely determined and the limit measure is unique. This ensures that the whole trajectory $n(t,\cdot)$ converges to the limit point given by~\eqref{e.prop} and~\eqref{alphaiIDE} as $t \to \infty$.

To prove claim~\eqref{rhosim} we proceed as follows. We note that the asymptotic result~\eqref{limitlap} now holds true for $J(t)$ and not for a mere subsequence $J(t_k)$. This implies that
\begin{equation}
\label{Integrated}
\int_0^t (r_M - \rho(s))\, \dd s \sim \frac{d}{2} \ln(t) \quad \text{as } t \to \infty.
\end{equation}

To conclude we only need to show that $\rho(t)$ has an asymptotic expansion of the form $a + \frac{b}{t} + o\left(\frac{1}{t}\right)$ as $t \to \infty$. The coefficients $a$ and $b$ are then necessarily $r_M$ and $-\frac{d}{2}$, owing to~\eqref{Integrated}. 

From now on, $C_1$ and $C_2$ will denote some generic real constants which might vary from line to line. 

We choose again $\e>0$ small enough so that $\bar{x}_i$ is the only maximum point of the function $r(x)$ in the ball $B(\bar{x}_i, \e)$ for every $i=1, \ldots, N$. Integrating over $B(\bar{x}_i, \e)$ both sides of the expression~\eqref{Implicit} for $n(t,x)$ we find that
$$
\rho(t) = e^{- \int_0^{t} \rho(s) \, \dd s} \sum_{i=1}^N \int_{B(\bar{x}_i, \e)} n^0(x) e^{r(x) t} \, \ddx \; + \; e^{- \int_0^{t} \rho(s) \, \dd s} \int_{\Omega \backslash \cup_{i=1}^N B(\bar{x}_i, \e)} n^0(x) e^{r(x) t} \, \ddx.
$$
The second term on the right-hand side of the above equation decays exponentially to~$0$ as $t \to \infty$, since $\rho(t) \to r_M$. Moreover, choosing $\e$ small enough so that -- under the additional assumption that the functions $n^0(x)$ and $r(x)$ are, respectively, of class $C^1$ and $C^3$ in a neighbourhood of each maximum point $\bar{x}_1, \ldots, \bar{x}_N$ -- we have $n^0 \in C^1\left(B(\bar{x}_i, \e) \right)$ and $r \in C^3\left(B(\bar{x}_i, \e) \right)$ for every $i=1, \ldots, N$, we can use the asymptotic expansion~\eqref{Laplace_further} and in so doing obtain the following asymptotic expression for the first integral on the right-hand side of the latter equation
$$
\int_{B(\bar{x}_i, \e)} n^0(x) e^{r(x) t} \, \ddx = \left(C_1 + \frac{C_2}{t} + o\left(\frac{1}{t}\right)\right) \frac{e^{r_M t}}{t^{\frac{d}{2}}} \quad \text{as } t \to \infty.
$$
Taken together, these results yield
\beq
\label{rhoJ}
\rho(t) = J(t) \left[ C_1 + \frac{C_2}{t} + o\left(\frac{1}{t}\right)\right] \quad \text{as } t \to \infty,
\eeq
with $J$ defined according to~\eqref{defJ}. 

We now prove that 
\beq
\label{Jnew}
J(t) = C_1 + \frac{C_2}{t} + o\left(\frac{1}{t}\right) \quad \text{as } t \to \infty,
\eeq
from which we can infer that $\rho(t)$ satisfies an asymptotic expansion of the same form, thus concluding the proof. In order to prove~\eqref{Jnew}, we notice that a sufficient condition for this to hold is that the function $u(t):= e^{\int_0^{t} \rho(s) \, ds}$ satisfies the estimate 
\beq
\label{newu}
u(t) = \frac{e^{r_M t}}{t^{d/2}} \left(C_1 + \frac{C_2}{t} + o\left(\frac{1}{t}\right)\right) \quad \text{as } t \to \infty.
\eeq
In order to prove~\eqref{newu}, we differentiate $u$ to obtain 
\beq
\label{newup}
u'(t) = \rho(t)  e^{\int_0^{t} \rho(s) \, ds} = \int_\Omega n^0(x) e^{r(x) t} \, \ddx = \frac{e^{r_M t}}{t^{d/2}} \left(C_1+ \frac{C_2}{t} + o\left(\frac{1}{t}\right)\right) \quad \text{as } t \to \infty,
\eeq
where the last equality has been established above. Since
$$
\int_{1}^t \frac{e^{s}}{s^{\alpha}}\, {\rm d}s= \frac{e^{ t}}{t^{\alpha}} \left(C_1 + \frac{C_2}{t} + o\left(\frac{1}{t}\right) \right)\quad \text{as } t \to \infty
$$
for any $\alpha > 0$, we can conclude that estimate~\eqref{newup} still holds true after integration and, therefore, estimate~\eqref{newu} is satisfied.
\end{proof}

\begin{remark} In the case where $d=1$, expression~\eqref{alphaiIDE} reads as
$$
a_i = A \frac{n^0(\bar{x}_i)}{\sqrt{|r''(\bar{x}_i)|}} \quad \text{for } \; i=1, \ldots, N.
$$
\end{remark}

\begin{remark}
The results of Theorem~\ref{Weight_IDE} can be extended to the case where some maximum points of the function $r(x)$ belong to the boundary $\partial \Omega$, a case that might be relevant for applications. In particular, letting $\partial \Omega$ be sufficiently smooth and using Laplace's method one can prove that if $\bar{x}_j \in \arg \max(r) \cap \partial \Omega$ is a stationary point of $r(x)$ (\textit{i.e.} $\nabla r(x) = 0$), then 
$$
a_j=  \frac{A}{2} \frac{n^0(\bar{x}_j)}{\sqrt{|\det(H_j)|}}.
$$
This implies that, all other things being equal, the weight in front of a Dirac mass centred at a boundary point will be half that of a Dirac mass centred at an interior point. On the other hand, if $\bar{x}_j \in \arg \max(r) \cap \partial \Omega$ is a nonstationary point of the restriction of $r(x)$ to the boundary, and there is at least one maximum point of $r(x)$ that belongs to the interior of $\Omega$, then the weight $a_j$ in front of $\delta_{\bar{x}_j}(x)$ will be zero (\textit{i.e.} the mass in a neighbourhood of $\bar{x}_j$ will vanish as $t \to \infty$). 
\end{remark}
\begin{remark}
The asymptotic relation~\eqref{rhosim} shows that the integral $\rho(t)$ of the solution to the IDE~\eqref{IDE} complemented with~\eqref{defR} converges to $r_M$ more slowly than the solution of the related logistic ordinary differential equation
$$
\frac{\dd N}{\dd t} = (r_M-N) N, \quad N \equiv N(t), \quad t \in (0,\infty),
$$
which converges exponentially to $r_M$ as $t \to \infty$. Moreover, whilst
$$
\sgn\left(\frac{\dd N(t)}{\dd t}\right) = \sgn\left(r_M-N(0)\right) \quad \forall t \in \mathbb{R}^+,  
$$
which means that $N(t)$ will approach the asymptotic value $r_M$ from below if $N(0)<r_M$ and from above if $N(0)>r_M$, the asymptotic relation~\eqref{rhosim} indicates that $\rho(t)$ will always approach $r_M$ from below independently from the value of $\rho(0)$. This also implies that if $\rho(0) > r_M$ then the derivative of $\rho(t)$ will change sign at least once on $\mathbb{R}^+$.

It is also worth comparing the speed of convergence given by the asymptotic relation~\eqref{rhosim} with the one that can be obtained based on the best estimates currently available in the literature, namely 
$$
\left(r^M - \rho(t)\right)^2 + \int_\Omega \left(r^M - r(x)\right) n(t,x)\, \ddx = \mathcal{O}\left(\frac{log(t)}{t}\right) \quad \text{as } t \to \infty.
$$
Furthermore, this cannot vanish as $\displaystyle{\mathcal{O}\left(\frac{1}{t^\alpha }\right)}$ for any $\alpha > 1$~\cite{Jabin2011, Pouchol2017b}.
\end{remark}

\subsection{Numerical solutions}
\label{IDEsimul}
To confirm the asymptotic results established by Theorem~\ref{Weight_IDE}, we solve numerically the Cauchy problem~\eqref{Basis}. In particular, we approximate the IDE~\eqref{IDE} complemented with~\eqref{defR} using the forward Euler method with step size $0.01$. We select a uniform discretisation of the interval $\Omega := [-1,2]$ consisting of $1000$ points as the computational domain of the independent variable $x$, and we consider $t \in [0,200]$. All numerical computations are performed in {\sc Matlab}. 

We choose the initial condition
\beq
\label{ICnumIDE}
n^0(x) \equiv \frac{2}{3} \quad \text{so that} \quad \int_{\Omega} n^0(x) \, \ddx = 2
\eeq
and we use the following definition 
\beq
\label{rnumIDE}
r(x) := e^{\frac{-\left(x+0.5\right)^2}{0.01}} + e^{\frac{-\left(x-1\right)^2}{0.1}},
\eeq
which satisfies the assumptions of Theorem~\ref{Weight_IDE}. As shown by the plot in Figure~\ref{Figr1}, the function $r(x)$ defined according to~\eqref{rnumIDE} has two maximum points, that is, $\bar{x}_1 \in [-1,0]$ and $\bar{x}_2 \in [0,2]$.
\begin{figure}[h!]
\begin{center}
\includegraphics[scale=0.4]{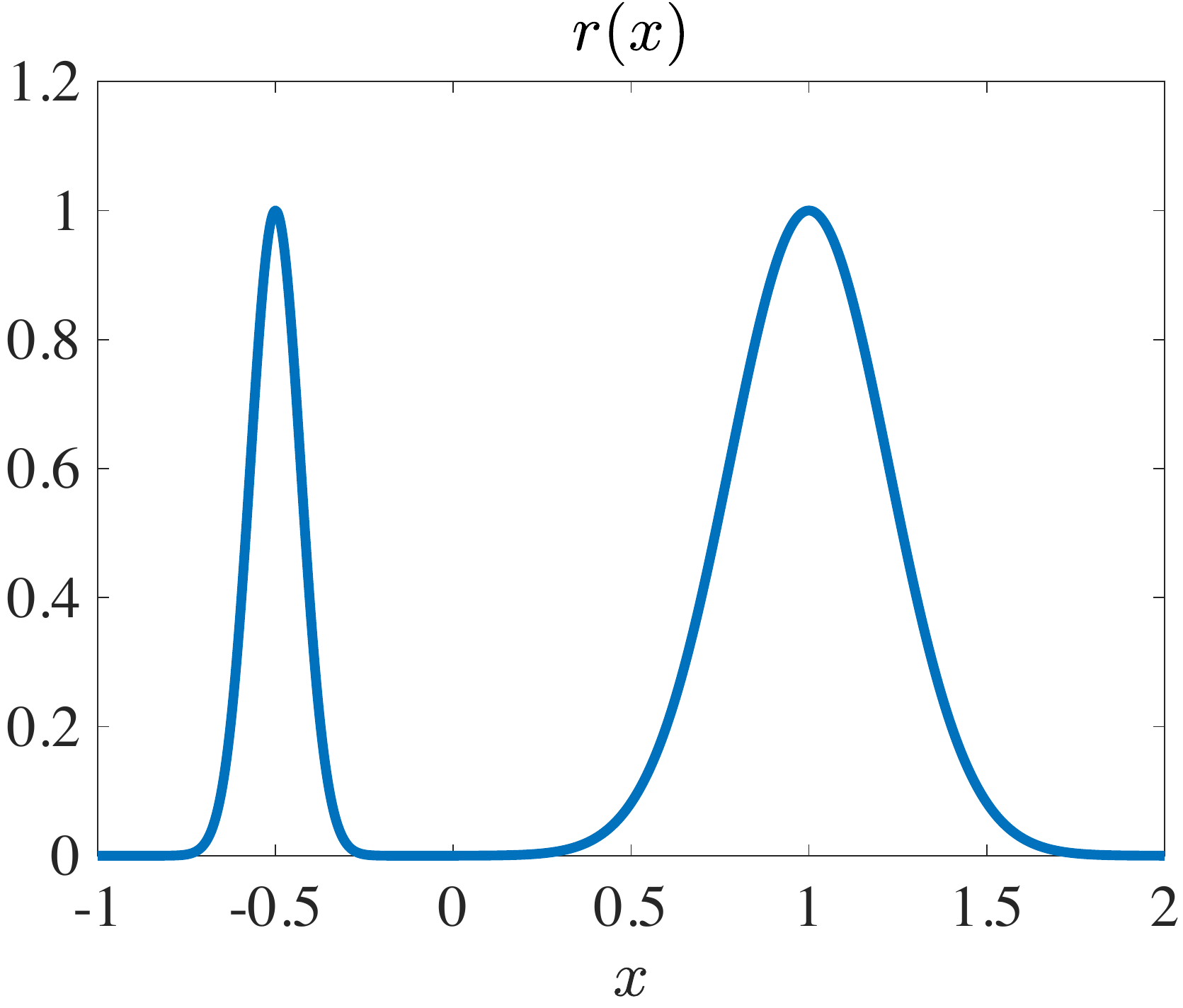}
\end{center}
\caption{Plot of the function $r(x)$ defined according to~\eqref{rnumIDE}.}
\label{Figr1}
\end{figure}

We compute numerically the integrals
\beq
\label{rho12IDE}
\rho_1(t) := \int_{-1}^{\frac{1}{2}} n(t,x) \, \ddx \quad \text{and} \quad \rho_2(t) := \int_{\frac{1}{2}}^2 n(t,x) \, \ddx
\eeq
and the coefficients $a_1$ and $a_2$ given by~\eqref{alphaiIDE}. In the one-dimensional setting considered here, the expressions given by~\eqref{alphaiIDE} read as   
\beq
\label{alphanumIDE}
a_1 = A \frac{n^0(\bar{x}_1)}{\sqrt{|r''(\bar{x}_1)|}} \; \text{ and } \; a_2 = A \frac{n^0(\bar{x}_2)}{\sqrt{|r''(\bar{x}_2)|}} \quad \text{with } \; A \; \text{ s.t. } \; a_1 + a_2 = 1.  
\eeq

The results obtained are summarised in Figure~\ref{Fig1} and Figure~\ref{Fig2}. As we would expect based on Theorem~\ref{Theorem1}, the numerical results displayed in Figure~\ref{Fig1} show that $n(t,x)$ becomes concentrated as a sum of two Dirac masses centred at the points $\bar{x}_1$ and $\bar{x}_2$ (left panel), while $\rho(t)$ converges to $r_M$ (right panel). 
\begin{figure}[h!]
\begin{center}
\includegraphics[scale=0.3]{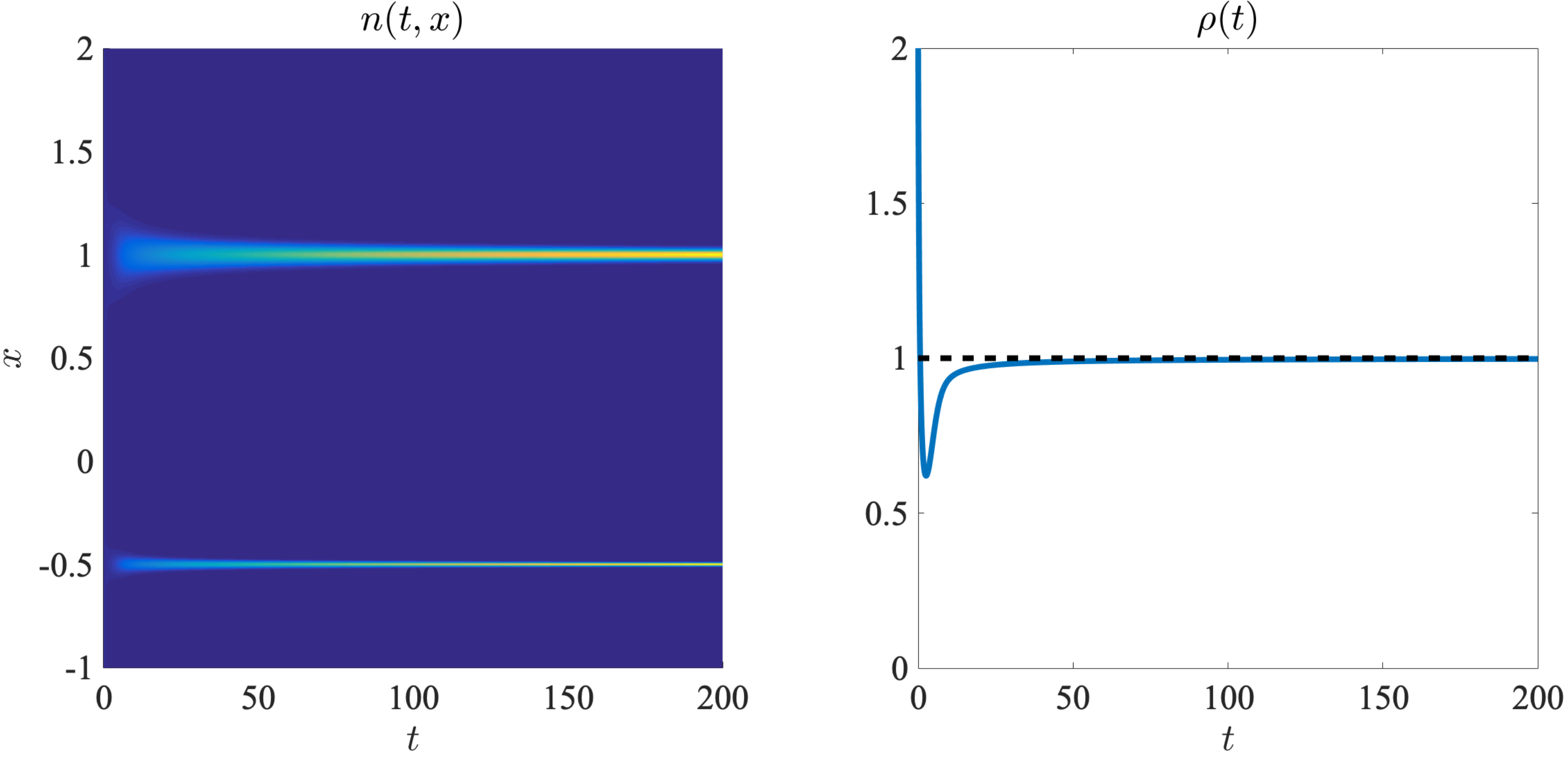}
\end{center}
\caption{Dynamics of $n(t,x)$ (left panel) and $\rho(t)$ (right panel) obtained by solving numerically the Cauchy problem~\eqref{Basis} with $n^0(x)$ and $r(x)$ defined according to~\eqref{ICnumIDE} and~\eqref{rnumIDE}. The black, dashed line in the right panel highlights the value of $r_M$.}
\label{Fig1}
\end{figure}

Furthermore, the curves displayed in the left panel of Figure~\ref{Fig2} show that, in agreement with the results established by Theorem~\ref{Weight_IDE}, 
the integrals $\rho_1(t)$ and $\rho_2(t)$ defined according to~\eqref{rho12IDE} converge, respectively, to the values $a_1 \, r_M$ and $a_2 \, r_M$, with $a_1$ and $a_2$ given by~\eqref{alphanumIDE}, while the curves in the right panel of Figure~\ref{Fig2} show that $\displaystyle{\left(r_M - \rho(t)\right) t  \to \frac{1}{2}}$ as $t \to \infty$, \emph{i.e.} the asymptotic relation~\eqref{rhosim} is verified.     
\begin{figure}[h!]
\begin{center}
\includegraphics[scale=0.3]{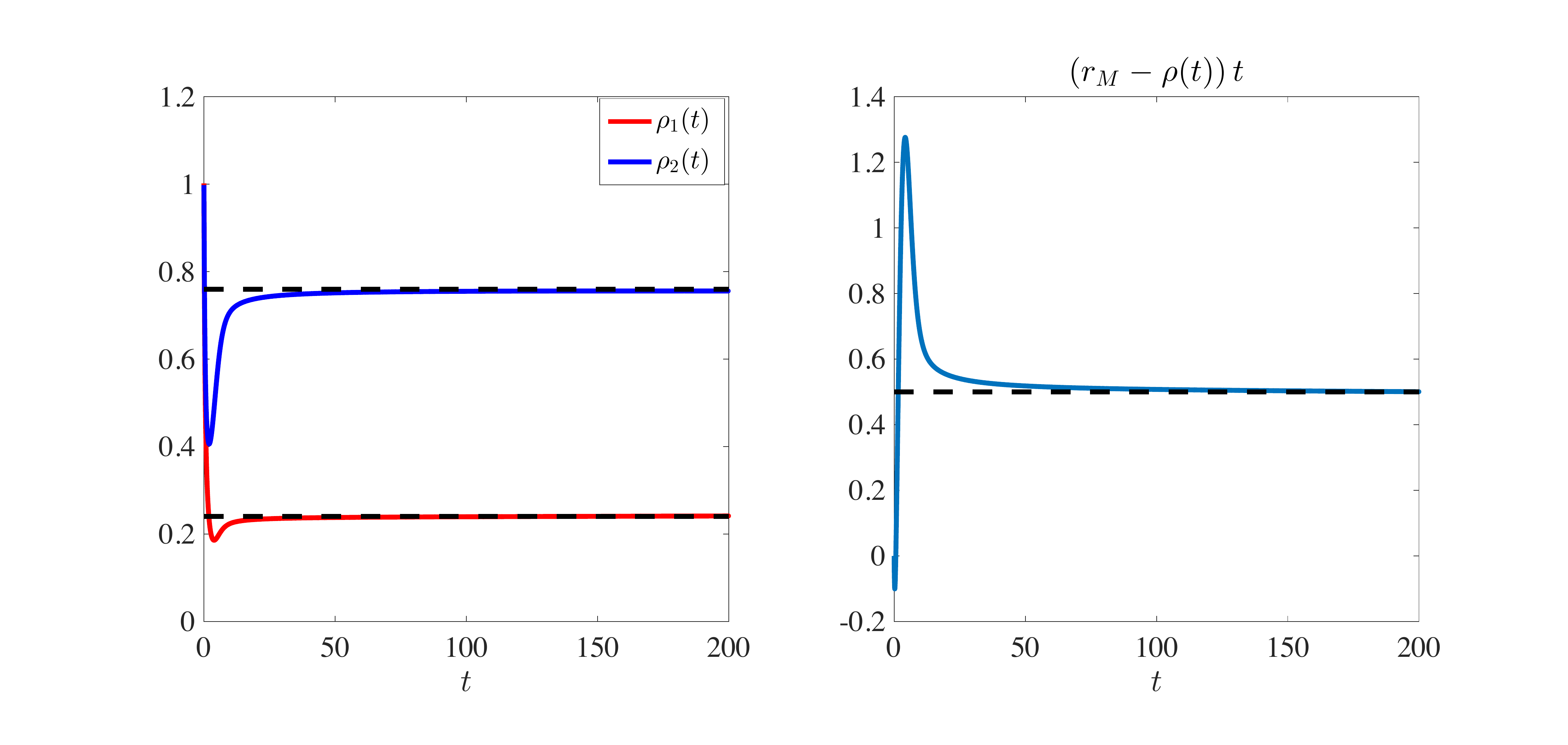}
\end{center}
\caption{Dynamics of the integrals $\rho_1(t)$ (red line) and $\rho_2(t)$ (blue line) defined according to~\eqref{rho12IDE} (left panel), and of the function $\left(r_M - \rho(t)\right) t$ (right panel). The black, dashed lines in the left panel highlight the values of the quantities $a_1 \, r_M$ and $a_2 \, r_M$,  with $a_1$ and $a_2$ given by~\eqref{alphanumIDE}, while the black, dashed line in the right panel corresponds to the asymptotic value of $\left(r_M - \rho(t)\right) t$ given by~\eqref{rhosim}. The integrals $\rho_1(t)$, $\rho_2(t)$ and $\rho(t)$ are computed using the numerical solution of the Cauchy problem~\eqref{Basis} subject to the initial condition~\eqref{ICnumIDE} with $r(x)$ defined according to~\eqref{rnumIDE}.}
\label{Fig2}
\end{figure}

\newpage
\section{Long-time behaviour of the initial-boundary value problem~\eqref{BasisMut}}
\label{secPDE}
In this section, we study the asymptotic behaviour of the solutions to the initial-boundary value problem~\eqref{BasisMut} as $t \to \infty$ (Section~\ref{PDEanalysis}), and we provide a sample of numerical solutions that confirm the analytical results obtained (Section~\ref{PDEsimul}). 

\subsection{Asymptotic analysis}
\label{PDEanalysis}
Under assumptions~\eqref{rM} there exists a unique non-negative classical solution $n_{\beta} \in C([0,+\infty), L^1(\Omega)) \cap C^1((0,+\infty), C^{2,\alpha} (\Omega))$ of the initial-boundary value problem~\eqref{BasisMut}~\cite{Coville2013}. Moreover, the behaviour of $n_{\beta}(t,x)$ in the asymptotic regime $t \to \infty$ is known to be governed by the principal eigenpair $(\lambda_{\beta},\psi_{\beta})$ of the elliptic differential operator $L_{\beta}$ (\emph{i.e.} $L_{\beta} \psi_{\beta} = - \lambda_{\beta} \psi_{\beta}$) defined according to~\eqref{Abeta} with zero Neumann boundary condition. The following theorem builds on the method of proof presented in~\cite{Coville2013, Leman2015} and extends previous results by characterising the speed of convergence of $\rho_\beta$ towards its limit as $t \to \infty$.
\\
\begin{theorem}
\label{Asymptotics}
Under assumptions~\eqref{rM}, the solution of the initial-boundary value problem~\eqref{BasisMut} is such that if $\lambda_\beta \geq 0$ then
\beq
\label{rhobetat0zero}
\rho_\beta(t)  \; \xrightarrow[t\rightarrow\infty] \; 0,
\eeq
whereas if $\lambda_\beta < 0$ then
\beq
\label{rhobetat0lambdabeta}
\rho_\beta(t) \; \xrightarrow[t\rightarrow\infty] \; -\lambda_\beta \quad \text{and} \quad n_\beta(t,\cdot) \; \xrightarrow[t\rightarrow\infty] \; -\lambda_\beta \, \psi_\beta \;\; \text{ in } \;\; L^\infty(\Omega).
\eeq
Furthermore, when $\lambda_\beta < 0$, if 
\beq
\label{intalpha2neq0}
\int_\Omega \alpha_2(n^0) \, \ddx \neq 0
\eeq
where $\alpha_2(n^0)$ is the the $L^2(\Omega)$-projection of $n^0$ onto the finite-dimensional eigenspace associated to $\lambda_{2, \beta}$, then there exist some real constants $K_1 \neq 0$ and $K_2~\neq~0$ such that
\beq
\label{rhobetasim}
-\lambda_\beta - \rho_\beta(t) \sim \left\{
      \begin{array}{ll}
        K_1 \, e^{\gamma_\beta t} \quad \text{if} \quad \gamma_\beta >  \lambda_\beta, \\
        K_2 \, e^{\lambda_\beta t} \quad \text{if} \quad \gamma_\beta\leq  \lambda_\beta, 
      \end{array}
    \right.
\quad \text{as }  t \to \infty, 
\eeq
with $\gamma_\beta$ being the opposite to the spectral gap of the operator $L_{\beta}$, which is defined via~\eqref{def:gammabeta}.
\end{theorem}

\begin{proof}
Let $u_\beta(t,x) := w(t) \, n_\beta(t,x)$ where $w(t)$ solves the Cauchy problem 
\beq
\label{CPa}
\begin{cases}
\displaystyle{w' = \left(\rho_\beta(t) + \lambda_\beta\right) \, w, \quad w \equiv w(t), \quad t \in (0,\infty)},
\\
w(0) = 1.
\end{cases}
\eeq
We have
\begin{eqnarray*}
\dfrac{\partial u_\beta}{\partial t} = w' \, n_\beta + w \, \dfrac{\partial n_\beta}{\partial t} &=& w' \, n_\beta + w \, \big[\big(r(x) - \rho_\beta(t)\big) n_\beta + \beta \, \Delta n_\beta\big] 
\\
&=& \left(w' -  \rho_{\beta}(t) w \right) n_{\beta} + w \big(r(x) n_\beta + \beta \, \Delta n_\beta \big)
\\
&=& \left(w' -  \rho_{\beta}(t) w \right) n_{\beta} + L_{\beta}[u_{\beta}].
\end{eqnarray*}
Using the fact that $w(t)$ solves the Cauchy problem~\eqref{CPa} and $n_{\beta}(t,x)$ is the solution of the initial-boundary value problem~\eqref{BasisMut}, we obtain the following initial-boundary value problem for $u_{\beta}(t,x)$
$$
\begin{cases}
\displaystyle{\frac{\partial u_{\beta}}{\partial t} - L_{\beta}[u_{\beta}] =  \lambda_\beta u_{\beta}, \quad u_{\beta} \equiv u_{\beta}(t,x), \quad (t,x) \in (0,\infty) \times \Omega},
\\\\
\nabla u_{\beta}(t,x) \cdot \nu(x) = 0, \quad (t,x) \in (0,\infty) \times \partial \Omega,
\\\\
u_{\beta}(0,x) = n^0(x), \quad  n^0 \in C(\Omega), \quad n^0 \geq 0, \quad n^0 \not\equiv 0.
\end{cases}
$$
This is a standard parabolic problem, the solution of which is such that
$$
u_\beta(t,\cdot) \; \xrightarrow[t\rightarrow\infty] \; K \, \psi_\beta \;\; \text{ in } \;\; L^\infty(\Omega) \;\; \text{ for some } \;\; K>0. 
$$
Since
$$
\int_{\Omega} u_\beta(t,x) \, \ddx = w(t) \, \rho_{\beta}(t) \quad \text{and} \quad \int_{\Omega} \psi_\beta(x) \, \ddx =1, 
$$
the above convergence result yields
\beq
\label{arhoconv}
w(t) \, \rho_\beta(t) \; \xrightarrow[t\rightarrow\infty] \; K.
\eeq
Hence the ordinary differential equation for $w(t)$ can be rewritten as
\beq
\label{ODEw}
w' = \lambda_\beta \left(\frac{K}{\lambda_\beta} +  w\right) + f(t), 
\eeq
with the function $f(t) := \rho_\beta(t) w(t) - K$ being such that $f(t) \to 0$ as $t \to \infty$. Using~\eqref{ODEw} we conclude that
$$
\text{if } \;\; \lambda_\beta > 0 \;\; \text{ then } \;\; \; w(t) \xrightarrow[t\rightarrow\infty] \; \infty, \;\; \text{whilst if } \;\; \lambda_\beta < 0 \;\; \text{ then } \;\; \; w(t) \xrightarrow[t\rightarrow\infty] \; \frac{K}{- \lambda_\beta}.
$$
These asymptotic results along with the asymptotic result~\eqref{arhoconv} ensure that
$$
\text{if } \;\; \lambda_\beta \geq 0 \;\; \text{ then } \;\; \; \rho_{\beta}(t) \xrightarrow[t\rightarrow\infty] \; 0, \;\; \text{whereas if } \;\; \lambda_\beta < 0 \;\; \text{ then } \;\; \; \rho_{\beta}(t) \xrightarrow[t\rightarrow\infty] \; -\lambda_\beta.
$$
Moreover, recalling that $u_\beta(t,x) = w(t) \, n_\beta(t,x)$ we find that if $\lambda_\beta < 0$ then
$$
n_\beta(t,\cdot) \; \xrightarrow[t\rightarrow\infty] \; -\lambda_{\beta} \, \psi_\beta \;\; \text{ in } \;\; L^\infty(\Omega). 
$$

We now turn our attention to estimate~\eqref{rhobetasim}. We recall that $\rho_\beta(t) = \frac{\int_{\Omega} u_\beta(t,x)\, \ddx}{w(t)}$, and we estimate the numerator and the denominator separately.

For the numerator, we expand $u_\beta$ further in the orthonormal basis associated to the operator $L_\beta$ and integrate to find that there exists some constant $C$ such that
$$
\int_{\Omega} u_\beta(t,x)\, \ddx = K + C e^{\gamma_\beta t} + o \left(e^{\gamma_\beta t}\right) \quad \text{ as } t \to \infty.
$$
If assumption~\eqref{intalpha2neq0} is satisfied then $C \neq 0$. The latter estimate for $\int_\Omega u_\beta(t,x) \, \ddx$ gives also a more detailed characterisation of the behaviour of the function $f(t)$ in~\eqref{ODEw} when $t \to \infty$, that is,
\beq
\label{fasy}
f(t) = \rho_\beta(t) w(t) - K \sim C e^{\gamma_\beta t} \quad \text{ as } t \to \infty.
\eeq

For the denominator, we note that solving~\eqref{ODEw} subject to the initial condition $w(0)=1$ gives
$$
w(t) = - \frac{K}{\lambda_\beta} + \left(1 + \frac{K}{\lambda_\beta}\right) e^{\lambda_\beta t} + \int_0^t f(s) e^{\lambda_\beta (t-s)} \, {\rm d}s.
$$
This along with estimate~\eqref{fasy} allows us to conclude that if $\gamma_\beta > \lambda_\beta$ then 
$$
w(t) = - \frac{K}{\lambda_\beta} +  \tilde C e^{\gamma_{\beta} t} + o\left( e^{\gamma_{\beta} t} \right) \quad \text{ as } t \to \infty,
$$
whereas if $\gamma_\beta \leq \lambda_\beta$ then
$$
w(t) = - \frac{K}{\lambda_\beta} +  \tilde C e^{\lambda_{\beta} t} + o\left( e^{\lambda_{\beta} t}\right) \quad \text{ as } t \to \infty,
$$
for some constant $\tilde C \neq 0$.

Combining these estimates for $\int_\Omega u_\beta(t,x) \, \ddx$ and $w(t)$ yields~\eqref{rhobetasim}, and concludes the proof of Theorem~\ref{Asymptotics}.
\end{proof}

\begin{remark}
Comparing the results established by Theorem~\ref{Asymptotics} [\emph{cf.} asymptotic relation~\eqref{rhobetasim}] with the results established by Theorem~\ref{Weight_IDE} [\emph{cf.} asymptotic relation~\eqref{rhosim}], one can see that there is a clear difference between the speed of convergence of $\rho_\beta(t)$ to its long-time limit $-\lambda_\beta$ and the speed of convergence of $\rho(t)$ to its long time-limit $r_M$.
\end{remark}

\begin{remark}
Some results on the behaviour of $\gamma_\beta$ in the asymptotic regime $\beta \to 0$ are available in the existing literature. In particular, it is known that typically $\gamma_\beta \sim e^{-C/\sqrt{\beta}}$ if $\beta$ tends to $0$~\cite{Simon1984}. Hence, when $\beta$ is small, the exponent in~\eqref{rhobetasim} will be exponentially small.
\end{remark}

The results established by Theorem~\ref{Asymptotics} show that analysing the long-time behaviour of the solution to the initial-boundary value problem~\eqref{BasisMut} comes down to studying the properties of the principal eigenpair $(\lambda_{\beta},\psi_{\beta})$. In particular, a general characterisation of $\lambda_{\beta}$ and $\psi_{\beta}(x)$ can be obtained when $\beta \to 0$. To illustrate this, we define $\beta := \e$, where $\e$ is a small positive parameter, and study the behaviour of $\lambda_{\e}$ and $\psi_{\e}(x)$ in the asymptotic regime $\e \to 0$. Proposition~\ref{Agreement} shows that the limit of $-\lambda_\e \psi_\e$ when $\e \to 0$ is given by a measure which has total mass equal to $r_M$ and consists of a weighted sum of Dirac masses centred at the points $\bar{x}_1, \ldots, \bar{x}_N$. This kind of concentration result is standard in semiclassical analysis (see for instance~\cite{Holcman2006}) but we give here a short proof that applies to our case for the sake of self-containedness.
\\
\begin{proposition}
\label{Agreement}
If assumptions~\eqref{rM} and~\eqref{Regularity2} are satisfied, then
\beq
\label{e.lambdaeps}
-\lambda_{\varepsilon} \; \xrightarrow[\e \rightarrow 0] \; r_M
\eeq
and, up to extraction of subsequences,
\beq
\label{e.propPDE}
\psi_\e(x)  \xrightharpoonup[\varepsilon  \rightarrow 0]{} \sum_{i=1}^N a_i \, \delta_{\bar{x}_i}(x) \;\; \text{ with } \;\; a_i \geq 0 \;\; \text{ and } \;\; \sum_{i=1}^N a_i =1.
\eeq 
\end{proposition}

\begin{proof}
We divide the proof of Proposition~\ref{Agreement} into two steps. We first prove claim~\eqref{e.lambdaeps} and then claim~\eqref{e.propPDE}. 
\\
\paragraph{Step 1: proof of~\eqref{e.lambdaeps}.}
Since $\lambda_{\e}$ is the principal eigenvalue of the differential elliptic operator $L_{\e}$, we have [\emph{cf.} equation~\eqref{lambdabeta}]   
\beq
\label{Rayeps}
\lambda_{\varepsilon} =\inf_{\phi \in{H^1(\Omega) \setminus \{0\}}} \mathcal{R}(L_{\e},\phi) \quad \text{with} \quad \mathcal{R}(L_{\e},\phi) = \frac{\e \displaystyle{\int_\Omega \left|\nabla \phi(x) \right|^2 \, \ddx - \int_\Omega r(x) \phi^2(x) \, \ddx}}{\displaystyle{\int_\Omega \phi^2(x) \, \ddx}}.
\eeq
We start by noting that
$$
\e \int_\Omega \left|\nabla \phi(x) \right|^2 \, \ddx - \int_\Omega r(x) \phi^2(x) \, \ddx \geq - r_M \int_\Omega \phi^2(x) \, \ddx \quad \forall \, \phi \in{H^1(\Omega)},
$$
and, therefore, $\lambda_{\varepsilon} \geq -r_M$ for all $\e>0$. Thus it suffices to show that $\lim\,  \lambda_{\varepsilon} \leq -r_M$ as $\e \to 0$ in order to prove~\eqref{e.lambdaeps}. To do this we construct a sequence of positive normalised $H^1$-functions $\phi_\varepsilon$ such that $\mathcal{R}(L_{\e},\phi_\varepsilon)$ converges to $-r_M$ as $\e \to 0$. 
We introduce the function 
\beq
\label{Gfun}
G: x \longmapsto C e^{-|x|^2},
\eeq
where $|\cdot|$ denotes the Euclidean norm on $\mathbb{R}^d$ and $C$ is a normalising constant such that $G$ has integral $1$. Recalling the classical result
$$
\frac{1}{\sigma^d} G\left(\frac{x-\bar{x}_i}{\sigma}\right) \xrightharpoonup[\sigma  \rightarrow 0]{} \delta_{\bar{x}_i}(x),
$$
we choose $\displaystyle{\phi_\varepsilon^2 : x \longmapsto \frac{1}{\varepsilon^{\frac{d}{4}}} G\left(\frac{x-\bar{x}_i}{\varepsilon^{\frac{1}{4}}}\right)}$ so that $\phi_\varepsilon^2 \xrightharpoonup[\e  \rightarrow 0]{} \delta_{\bar{x}_i}$. Using the fact that
$$
\int_\Omega r(x) \phi_\varepsilon^2(x) \, \ddx \; \xrightarrow[\e \rightarrow 0] \; r_M, \quad \int_\Omega \phi_\varepsilon^2(x) \, \ddx \; \xrightarrow[\e \rightarrow 0] \; 1
$$
and
$$
\varepsilon \int_\Omega \left|\nabla \phi_\varepsilon(x) \right|^2 \, \ddx = \int_\Omega \left|x-\bar{x}_i\right|^2 \phi_{\varepsilon}^2(x) \, \ddx  \; \xrightarrow[\e \rightarrow 0] \; \left|\bar{x}_i-\bar{x}_i\right|^2 = 0,
$$
we obtain
$$
\mathcal{R}(L_{\e},\phi_{\e})  \; \xrightarrow[\e \rightarrow 0] \; - r_M.
$$
This concludes the proof of~\eqref{e.lambdaeps}.
\\
\paragraph{Step 2: proof of~\eqref{e.propPDE}.} 
The pair $(\lambda_{\varepsilon},\psi_{\varepsilon})$ satisfies the eigenvalue problem
$$
\begin{cases}
\displaystyle{- L_{\e} \, \psi_{\varepsilon} =  \lambda_{\e} \, \psi_{\varepsilon}, \quad \text{in } \; \Omega},
\\
\nabla \psi_{\varepsilon} \cdot \nu = 0, \qquad\;\;\;\;\; \text{on } \;  \partial \Omega,
\end{cases}
$$
and integrating over $\Omega$ we find
$$
\int_\Omega \left(- \lambda_\e - r(x)\right) \, \psi_\e(x) \, \ddx = 0.
$$
Hence,
\beq
\label{e:concpsie}
\int_\Omega \left(r_M - r(x)\right) \, \psi_\e(x) \, \ddx \, \xrightarrow[\e \rightarrow 0] \, 0.
\eeq
Finally, for any  $\varphi \in C\big(\overline{\Omega}\big)$ with $\supp(\varphi) \cap \arg \max(r) = \emptyset$ we have
\begin{eqnarray*}
\left|\int_{\Omega}\varphi(x) \,  \psi_{\varepsilon}(x) \, \ddx \right| &=& \left|\int_{\Omega} \frac{\varphi(x)}{r_M - r(x)} \, \left(r_M - r(x) \right) \,  \psi_{\varepsilon}(x) \, \ddx\right| 
\\
&\leq& \max_{x \in \supp(\varphi)} \left|\left(\frac{\varphi(x)}{r_M - r(x)}\right)\right| \int_\Omega \left(r_M - r(x)\right) \, \psi_\e(x) \, \ddx.
\end{eqnarray*}
The latter integral inequality along with the asymptotic result~\eqref{e:concpsie} yields
$$
\int_{\Omega}\varphi(x) \,  \psi_{\varepsilon}(x) \, \ddx \xrightarrow[\e \rightarrow 0] \, 0 \quad \forall \, \varphi \in C\big(\overline{\Omega}\big) \quad \text{s.t.} \quad \supp(\varphi) \cap \arg \max(r) = \emptyset.
$$
This concludes the proof of~\eqref{e.propPDE}.
\end{proof}

\begin{remark}
The results of Proposition~\ref{Agreement} imply that if $\arg \max(r) = \left\{\bar{x}_1 \right\}$ then the limit of $-\lambda_\e \psi_{\e}$ when $\e \to 0$ is given by the measure $r_M \, \delta_{\bar{x}_1}(x)$.
\end{remark}

\begin{remark}
The proof of Proposition~\ref{Agreement} can be adapted to the case where the function $r(x)$ attains its maximum only on the boundary $\partial \Omega$. We expect that this can be done, as in Laplace's method, by adjusting the normalising constant $C$ in~\eqref{Gfun}, depending on the nature of the maximum at the boundary (stationary or not). However, we consider here only the simpler case corresponding to Proposition~\ref{Agreement}, which suffices for our purposes.
\end{remark}

\begin{remark}
Since $r_M>0$, expression~\eqref{Rayeps} for the Rayleigh quotient is such that if $\e$ is small enough then $-\lambda_\e > 0$. Hence, based on the result of Theorem~\ref{Asymptotics}, we have that $\rho_{\e}(t)$ will not vanish as $t \to \infty$ when $\e$ is sufficiently small.
\end{remark}

In the framework of the results established by Proposition~\ref{Agreement}, to fully characterise the long-time limit of $n_{\e}(t,x)$ when $\e \to 0$ it is necessary to assess whether there exists a unique set of admissible coefficients $a_1, \ldots, a_N$ (\emph{i.e.} if the limit measure is unique); if so, one needs to identify the values of the coefficients that define the only admissible limit measure.

A case where we expect the limit measure to be unique is when the set $\Omega$ and the function $r(x)$ are symmetric with respect to the hyperplane $S$ defined according to~\eqref{defS}. In this case, a complete characterisation of the limit measure is given by the following proposition.

\begin{proposition}
\label{propsym}
Under assumptions~\eqref{rM} and~\eqref{Regularity2}, letting $N=2$ and making the additional symmetry assumptions~\eqref{Omegasym} and~\eqref{rsym}, we have
\beq
\label{withsym}
\psi_{\e}(x)  \xrightharpoonup[\e  \rightarrow 0]{} \frac{1}{2} \big(\delta_{\bar{x}_1}(x) + \delta_{\bar{x}_2}(x)\big).
\eeq 
\end{proposition}

\begin{proof}
Under the symmetry assumptions~\eqref{Omegasym} and~\eqref{rsym} the points $\bar{x}_1 \in \Omega$ and $\bar{x}_2 \in \Omega$ are symmetric with respect to the hyperplane $S$, \emph{i.e.}
$$
\text{if} \quad \bar{x}_1 = (\bar{x}_{1 \,1}, \bar{x}_{1 \,2}, \ldots, \bar{x}_{1 \,d}) \quad \text{then necessarily} \quad \bar{x}_2 = (-\bar{x}_{1 \,1}, \bar{x}_{1 \,2}, \ldots, \bar{x}_{1 \,d}).
$$
Moreover, in the case where $N=2$, the result established by Proposition~\ref{Agreement} implies that
$$
\psi_{\e}(x)  \xrightharpoonup[\e  \rightarrow 0]{} a \, \delta_{\bar{x}_1}(x) + (1-a) \, \delta_{\bar{x}_2}(x), \quad \text{for some } \; a \geq 0.
$$
Finally, Lemma~\ref{Eigen} ensures that $\psi_{\e}$ is symmetric with respect to the hyperplane $S$, and Lemma~\ref{Symmetry} in turn ensures that the weak limit point of $\psi_{\e}$ for $\e \to 0$ is symmetric with respect to the hyperplane $S$ as well. Hence, $a=\frac{1}{2}$ and there is a unique limit point given by
$$
\frac{1}{2} \, \delta_{\bar{x}_1}(x) + \frac{1}{2} \, \delta_{\bar{x}_2}(x).
$$
Since the sequence $\psi_{\e}$ is bounded in $L^1(\Omega)$, the Banach-Alaoglu Theorem ensures that it is relatively (weakly-$*$) compact in $\mathcal{M}^1(\overline\Omega)$. This along with the uniqueness of the limit point gives the convergence of the whole sequence $\psi_{\e}$, \emph{i.e.}
$$
\psi_{\e}(x)  \xrightharpoonup[\e  \rightarrow 0]{} \frac{1}{2} \big(\delta_{\bar{x}_1}(x) + \delta_{\bar{x}_2}(x)\big),
$$
which concludes the proof of Proposition~\ref{propsym}.
\end{proof}

Furthermore, an almost exhaustive characterisation of the limit measure in the absence of particular symmetries is provided by the following proposition, whereby the function $\zeta : \arg \max \left(r\right) \to \mathbb{R}^+$, which is defined as
\beq
\label{armgminzeta}
\zeta(\bar{x}_i) := \sum_{j=1}^{d} \sqrt{|\lambda_j^i|}
\eeq 
where $(\lambda_{j}^i)_{1 \leq j \leq d}$ are the eigenvalues of $H_i$ (each counted with its multiplicity), is used to characterise the concavity of the function $r(x)$ at the maximum points, as it was done in previous papers on semiclassical analysis~\cite{Helffer1984, Helffer1985}.

\begin{proposition}
\label{propnosym}
Under assumptions~\eqref{rM} and~\eqref{Regularity2}, 
\beq
\label{conc_further}
\psi_{\e} \text{ concentrates on the set } \arg \min \left(\zeta\right) \text{ as } \e \to 0.
\eeq
In particular, if
$\arg\min  \left(\zeta \right) = \left\{\bar{x}_m \right\}$ for some $1 \leq m \leq N$, then 
\beq
\label{singleton}
\psi_{\e}(x)  \xrightharpoonup[\e  \rightarrow 0]{} \delta_{\bar{x}_m}(x).
\eeq 
\end{proposition}

\begin{proof}
We note that studying the asymptotic behaviour of $\psi_{\varepsilon}$ when $\e \to 0$ is equivalent to studying the asymptotic behaviour of the principal eigenfunction $\varphi_{\varepsilon}$ of the differential elliptic operator $\varepsilon \Delta - V$ with $V := -r$. The result of Proposition~\ref{Agreement} ensures that the support of the weak limit of $\varphi_{\varepsilon}$ as $\e \to 0$ will be a (possibly improper) subset of the set $\arg \min\left(V\right) = \left(\bar{x}_1, \ldots, \bar{x}_N\right)$. Investigating at which points of this discrete set the weak limit point of the sequence $\varphi_{\varepsilon}$ will actually be concentrated is a fundamental question in semiclassical analysis. Such a question arises in the study of the dynamics of a particle confined within a region of space surrounding a minimum point of the potential $V$ (\emph{i.e.} a potential well) in the asymptotic regime of small noise (\emph{i.e.} when $\e \to 0$)~\cite{Simon1983, Helffer1984, Helffer1985, Holcman2006}. Recasting the problem in this way, we can use the asymptotic results presented in~\cite{Helffer1984, Helffer1985} which ensure that, under assumptions~\eqref{rM} and~\eqref{Regularity2}, when $\varepsilon$ tends to $0$, the principal eigenfunction $\varphi_{\varepsilon}$ concentrates on the set $\displaystyle{\arg\min (\zeta)}$, with $\zeta$ defined via~\eqref{armgminzeta}. Hence, under the additional assumption that the set $\displaystyle{\arg\min (\zeta)}$ coincides with the singleton $\left\{\bar{x}_m \right\}$ for some $1 \leq m \leq N$, we find that $\varphi_{\varepsilon}$ concentrates at the point $\bar{x}_m$ as $\e \to 0$, whence~\eqref{singleton}.
\end{proof}

\begin{remark} 
In the case where $d=1$, the assumption $\arg\min  \left(\zeta \right) = \left\{\bar{x}_m \right\}$ reads as
\beq
\label{armgminzeta1d}
\left\{\bar{x}_m \right\} = \arg\min_{\bar{x}_i \in \arg \max r} |r''(\bar{x}_i)|.
\eeq 
\end{remark}
\begin{remark} 
Comparing the results of Proposition~\ref{propnosym} with the results established by Theorem~\ref{Weight_IDE} one can see that there is a stark difference between the long-time behaviour of the solution to the initial-boundary value problem~\eqref{BasisMut} for $\beta \to 0$ and the long-time behaviour of the solution to the Cauchy problem~\eqref{Basis}. In fact, in the case where the set $\arg\min \left(\zeta\right)$ is reduced to a singleton and $n^0(x) > 0$ for all $x \in \arg \max(r)$, Theorem~\ref{Weight_IDE} shows that the long-time limit of $n(t,x)$ will be given by a sum of multiple Dirac masses with different positive weights, whereas Proposition~\ref{propnosym} shows that the long-time limit of $n_{\beta}(t,x)$ for $\beta \to 0$ will consist of one single Dirac mass. From the point of view of evolutionary dynamics, this implies that, all else being equal, individuals in the phenotypic states $x \in \arg \max(r)$ will coexist in the absence of phenotypic changes, whereas only individuals in the phenotypic state $\bar{x}_m$ will ultimately survive when heritable phenotypic changes occur. This provides a mathematical formalisation of the idea that, while being historically assumed to play a neutral role in evolutionary outcome, heritable phenotypic changes can shape the equilibrium phenotypic distribution of asexual populations with multi-peaked fitness landscapes, even if there is no bias in the generation of novel phenotypic variants.
\end{remark}

\begin{remark}
We remark that when the set $\displaystyle{\arg\min \left(\zeta\right)} $ is not a singleton, it is still possible to go further in reducing the support of the limit point of the sequence $\psi_{\varepsilon}$ as $\e \to 0$. However, the conditions determining which of the coefficients $a_i$ will be different from zero become rather convoluted, as shown by the results of semiclassical analysis presented in~\cite{Holcman2006}. Therefore, we consider here only the simpler case corresponding to Proposition~\ref{propnosym}, which suffices for our purposes.
\end{remark}

\subsection{Numerical solutions}
\label{PDEsimul}
To confirm the asymptotic results established by Propositions~\ref{Agreement}-\ref{propnosym}, we solve numerically the initial-boundary value problem~\eqref{BasisMut}. Numerical solutions are constructed by approximating the diffusion term via a second-order central difference scheme~\cite{leveque2007finite} and then using the forward Euler method with step size $0.01$ to approximate the resulting system of ordinary differential equations.  We select a discretisation of the interval $\Omega := [-1,2]$ consisting of $1000$ points as the computational domain of the independent variable $x$ and let $t \in [0,t_f]$ with $t_f$ being either $200$ or $800$. All numerical computations are performed in {\sc Matlab}. 

We define $\beta=10^{-6}$, choose the initial condition~\eqref{ICnumIDE}, and use either the following definition 
\beq
\label{rnumPDE}
r(x) := e^{\frac{-\left(x+0.5\right)^2}{0.01}} + e^{\frac{-\left(x-1.5\right)^2}{0.01}}
\eeq
or definition~\eqref{rnumIDE}. Definition~\eqref{rnumPDE} satisfies the assumptions of Proposition~\ref{propsym} with $S~:=~\{0.5 \}$ (\emph{cf.} the plot in Figure~\ref{Figr2}), whereas definition~\eqref{rnumIDE} satisfies the assumptions of Proposition~\ref{propnosym} and, as previously noted, it has two maximum points $\bar{x}_1 \in [-1,0]$ and $\bar{x}_2 \in [0,2]$ (\emph{cf.} the plot in Figure~\ref{Figr1}).
\begin{figure}[h!]
\begin{center}
\includegraphics[scale=0.4]{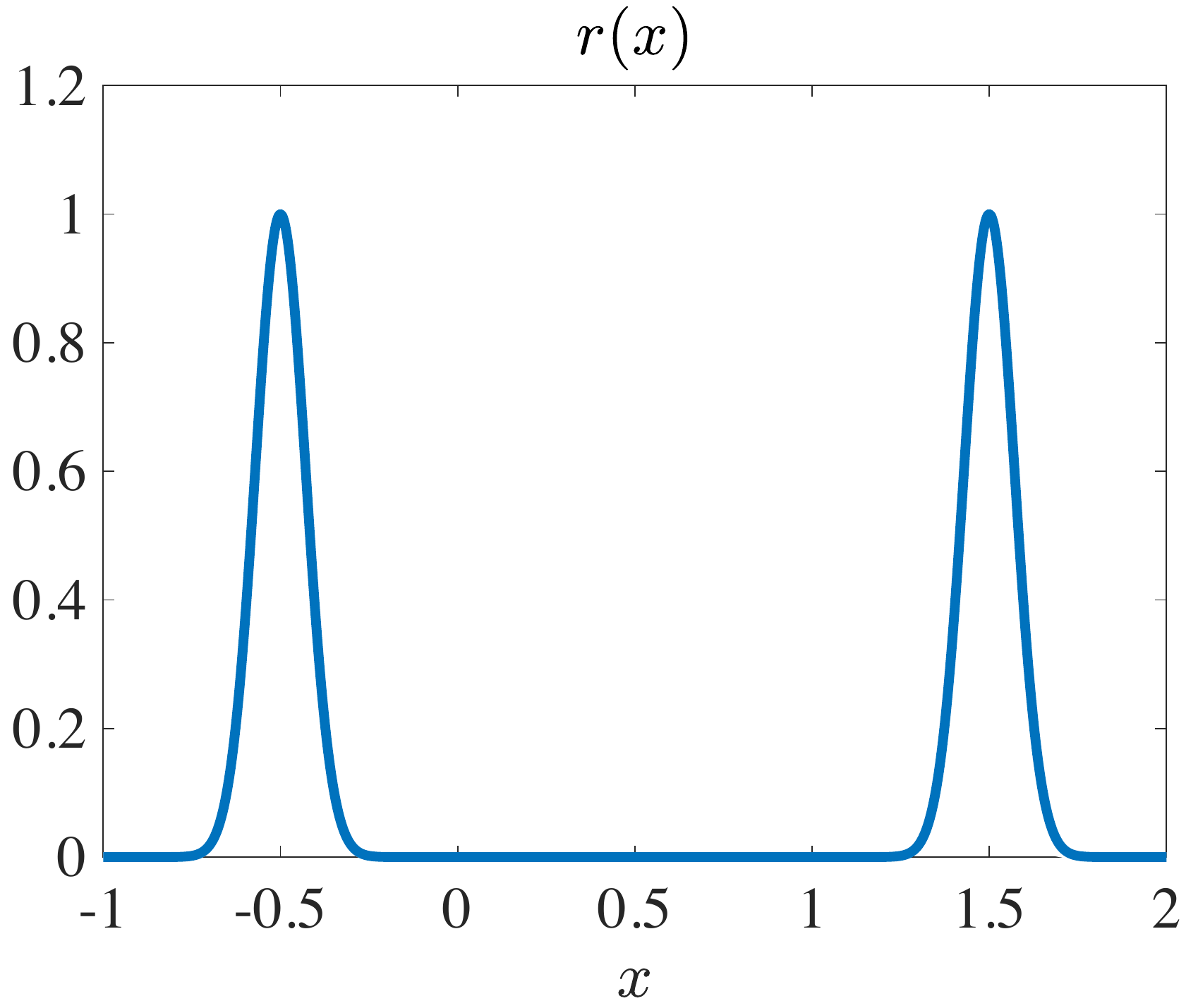}
\end{center}
\caption{Plot of the function $r(x)$ defined according to~\eqref{rnumPDE}.}
\label{Figr2}
\end{figure}

We compute numerically the following integrals
\beq
\label{rho12PDE}
\rho_1(t) := \int_{-1}^{\frac{1}{2}} n_{\beta}(t,x) \, \ddx \quad \text{and} \quad \rho_2(t) := \int_{\frac{1}{2}}^2 n_{\beta}(t,x) \, \ddx.
\eeq

The results obtained are summarised in Figure~\ref{Fig3} and Figure~\ref{Fig4}. As we would expect based on Proposition~\ref{Agreement} and Proposition~\ref{propsym}, the numerical results displayed in Figure~\ref{Fig3} show that when $r(x)$ is defined according to~\eqref{rnumPDE} the solution $n_{\beta}(t,x)$ becomes concentrated as a sum of two Dirac masses centred at the points $\bar{x}_1$ and $\bar{x}_2$, the integral $\rho_{\beta}(t)$ converges to $r_M$, and the integrals $\rho_1(t)$ and $\rho_2(t)$ defined via~\eqref{rho12PDE} both converge to $\displaystyle{\frac{r_M}{2}}$. 
\begin{figure}[h!]
\begin{center}
\includegraphics[scale=0.3]{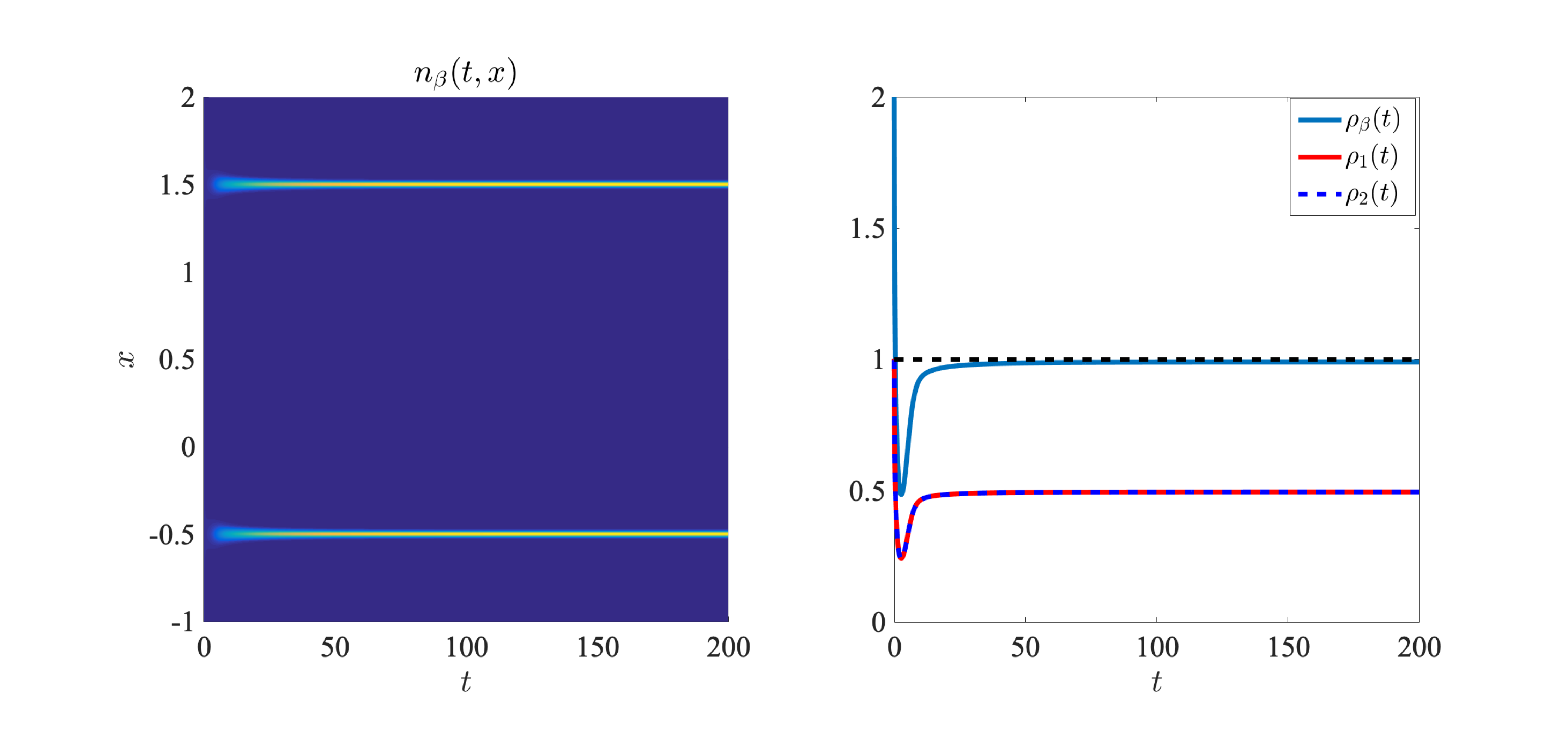}
\end{center}
\caption{Dynamics of $n_{\beta}(t,x)$ (left panel) and $\rho_\beta(t)$ (right panel) obtained by solving numerically the initial-boundary value problem~\eqref{BasisMut} with $\beta=10^{-6}$ and with $n^0(x)$ and $r(x)$ defined according to~\eqref{ICnumIDE} and~\eqref{rnumPDE}. The black, dashed line in the right panel highlights the value of $r_M$, while the red line and the blue, dashed line correspond, respectively, to the integrals $\rho_1(t)$ and $\rho_2(t)$ defined according to~\eqref{rho12PDE}.}
\label{Fig3}
\end{figure}

\begin{figure}[h!]
\begin{center}
\includegraphics[scale=0.3]{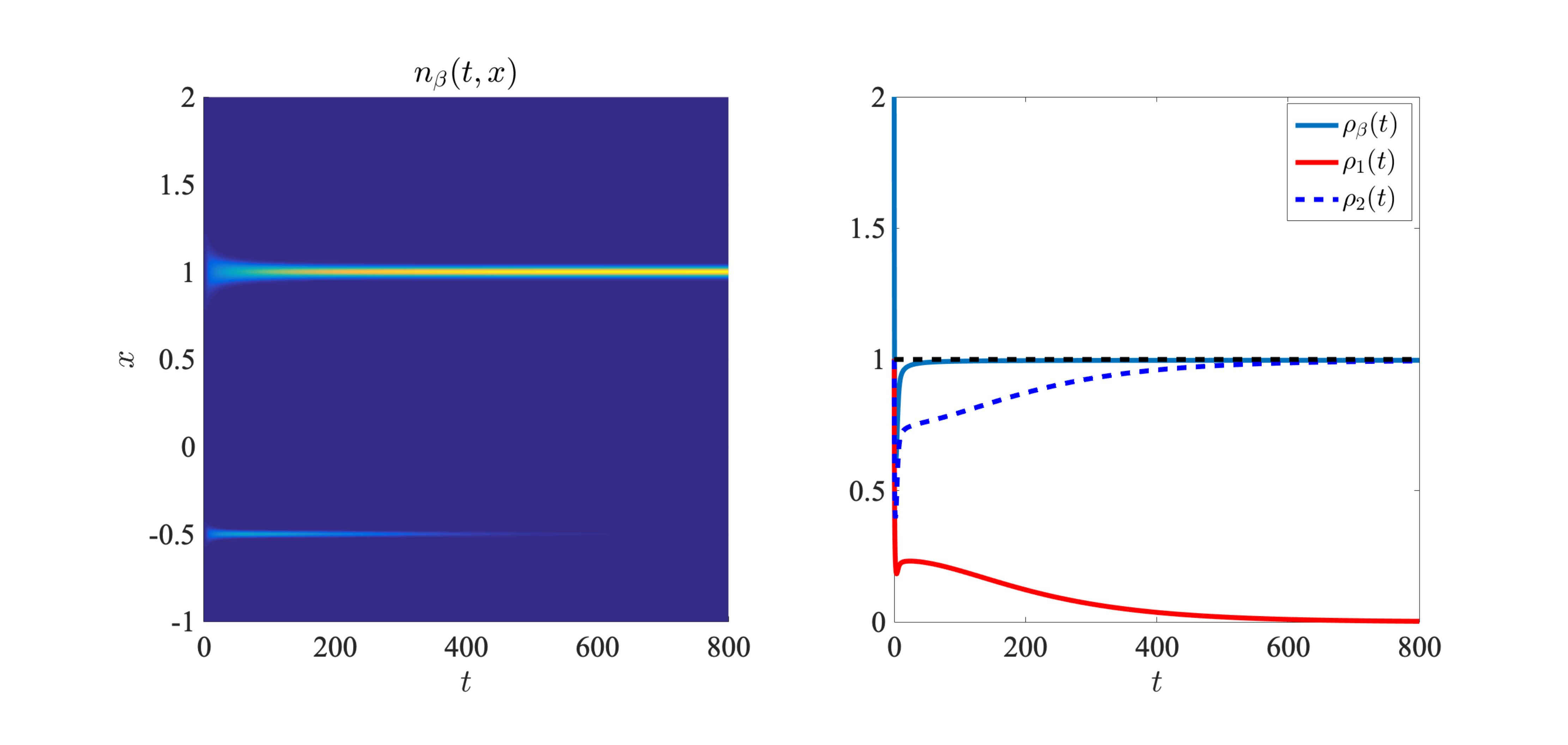}
\end{center}
\caption{Dynamics of $n_{\beta}(t,x)$ (left panel) and $\rho_{\beta}(t)$ (right panel) obtained by solving numerically the initial-boundary value problem~\eqref{BasisMut} with $\beta=10^{-6}$ and with  $n^0(x)$ and $r(x)$ defined according to~\eqref{ICnumIDE} and~\eqref{rnumIDE}. The black line in the right panel highlights the value of $r_M$, while the red line and the blue, dashed line correspond, respectively, to the integrals $\rho_1(t)$ and $\rho_2(t)$ defined according to~\eqref{rho12PDE}.}
\label{Fig4}
\end{figure}

On the other hand, the numerical results displayed in Figure~\ref{Fig4} show that, in agreement with the results of Proposition~\ref{Agreement} and Proposition~\ref{propnosym}, when $r(x)$ is defined according to~\eqref{rnumIDE} the integral $\rho_{\beta}(t)$ converges to $r_M$ while the solution $n_{\beta}(t,x)$ becomes concentrated as one single Dirac mass centred at the point $\bar{x}_2$ (left panel), which is the maximum point of the function $r(x)$ that satisfies condition~\eqref{armgminzeta1d} -- \emph{i.e.} $|r''(\bar{x}_2)| < |r''(\bar{x}_1)|$. As a consequence, the integral $\rho_1(t)$ converges to zero, whereas the integral $\rho_2(t)$ converges to $r_M$.

\newpage
\section{Research perspectives}
\label{secresper}
There are several possible generalisations of the prototypical selection model~\eqref{IDE} and selection-mutation model~\eqref{PDE} for which suitable developments of the methods used here would be relevant.

\paragraph{More general saturating non-local terms.} 
A natural way to extend our study would be to replace the saturating term $\rho(t)$ with a more general non-local term of the form $\int_\Omega K(x,y) \, n(t,y) \, {\rm d}y$, where the kernel $K(x,y)$ models the effect of competitive interactions between individuals in the phenotypic state $x$ and other individuals in a generic phenotypic state $y$. While the long-time behaviour of the IDE model with such a more general saturating non-local term was extensively studied in~\cite{Jabin2011}, where the convergence of the solution to a weighted sum of Dirac masses was also investigated depending on the properties of the kernel $K$, the existing literature still lacks a precise characterisation of the long-time behaviour of the solutions for the corresponding PDE model, with the exception of the particular case when $K(x,y) \equiv k(y)$ or similar cases~\cite{Coville2013}. While we expect that extending our results to this particular case would be relatively easy, the case of a generic kernel $K(x,y)$ is an open problem that requires a different approach compared to the one undertaken here. 

\paragraph{Integral kernel modelling phenotypic changes.}
Our results could be extended to the case where the linear diffusion operator in the non-local PDE~\eqref{PDE} is replaced by an integral term of the form $\int_\Omega \left(M(x,y) \, n(t,y) - M(y,x) \, n(t,x) \right)\, {\rm d}y$, where the kernel $M(x,y)$ models the transition of individuals from a generic phenotypic state $y$ to the phenotypic state $x$. In~\cite{Bonnefon2015} it was shown that, when phenotypic changes are modelled through such an integral kernel, the solution of selection-mutation models like the one considered here will typically converge to a measure as $t \to \infty$, and a criterion was derived to determine whether the limit measure would be singular or absolutely continuous. Suitable developments of our methods would make it possible to investigate the dependence of such a criterion on the weight of mutations compared to selection, which would be captured by a scaling parameter analogous to our parameter $\e$.

\paragraph{Systems of equations.} 
It would also be interesting to extend our results to the case of systems of IDEs of the form of~\eqref{IDE} and systems of non-local PDEs of the form of~\eqref{PDE}. In this regard, the results presented in~\cite{Pouchol2017b} for a specific system of IDEs could prove useful, since they establish the convergence of the solution to a measure as $t \to \infty$ and provide a characterisation of its support. As for systems of corresponding non-local PDEs, the convergence of the components of the solution to the principal eigenfunctions of the related elliptic differential operators when $t \to \infty$ has been proved for a two-by-two competitive system~\cite{Leman2015}. Apart from these particular cases, the long-time behaviour of the solutions of these systems of IDEs and non-local PDEs is still an open problem, which requires a different approach compared to the one undertaken here. 

\section*{Acknowledgments}
The authors would like to thank Sepideh Mirrahimi for the interesting discussions, and gratefully acknowledge Bernard Helffer for a fruitful exchange of emails regarding his semiclassical analysis results obtained in collaboration with Johannes Sj\"ostrand~\cite{Helffer1984, Helffer1985}. C.P. acknowledges support from the Swedish Foundation of Strategic Research grant AM13-004. T.L. acknowledges support of the project PICS-CNRS no. 07688.

\section*{Appendix. Proof of Theorem \ref{Theorem1} under assumption~\eqref{rm}}
\label{Appendix1}
We start by noting that $\rho(t)$ is uniformly bounded in $L^{\infty}([0,\infty))$. In fact, integrating both sides of the IDE for $n(t,x)$ over $\Omega$ and using Gr\"onwall's lemma yields 
\beq
\label{e.lbide}
\frac{\dd \rho}{\dd t} \geq (r_m- \rho) \rho \quad \Longrightarrow \quad \rho(t) \, \geq \, \min\left(r_m, \rho(0)\right) =: \rho_m \quad \forall \, t \in [0,\infty)
\eeq
and
\beq
\label{e.ubide}
\frac{\dd \rho}{\dd t} \leq (r_M - \rho) \rho \quad \Longrightarrow \quad \rho(t) \, \leq \, \max\left(r_M, \rho(0)\right) =: \rho_M \quad \forall \, t \in [0,\infty).
\eeq
Then we prove that $\rho \in BV([0,+\infty))$. In order to do this, we define 
$$
q:= \frac{\dd \rho}{\dd t} = \int_\Omega (r(x)- \rho) \, n(t,x) \, \ddx
$$
so that differentiating we obtain 
$$
\frac{\dd q}{\dd t} = \int_\Omega (r(x)- \rho)^2 n(t,x) \, \ddx - q \rho. 
$$
Multiplying both sides of the latter differential equation by $\displaystyle{- \big(\sgn (q)\big)_-}$ and estimating the right-hand side of the resulting differential equation from above we find that
\beq
\label{e.ineq1}
\frac{\dd q_-}{\dd t} \leq - \rho_m \, q_- \quad \Longrightarrow \quad \left(q(t)\right)_- \leq \left(q(0)\right)_- \, e^{- \rho_m \, t} \quad \forall \, t \in [0,\infty).
\eeq
Moreover, for any $T>0$ we have
\beq
\label{e.ineq2}
\int_0^T q(t) \, \dd t \; = \; \rho(T) - \rho(0) \; \leq \; \rho_M.
\eeq
Using estimates~\eqref{e.ineq1} and~\eqref{e.ineq2} we obtain
$$
\int_0^T \left(\frac{\dd \rho}{\dd t}\right)_+ \, \dd t \; = \; \int_0^T \frac{\dd \rho}{\dd t} \, \dd t  \; + \; \int_0^T \left(\frac{\dd \rho}{\dd t}\right)_- \, \dd t \; < \; \infty 
$$
and letting $T \to \infty$ we find
$$
\int_0^{\infty} \left(\frac{\dd \rho}{\dd t}\right)_+ \dd t \; < \; \infty.
$$
This estimate along with the fact that $\rho \in L^{\infty}([0,\infty))$ ensures that $\rho \in BV([0,+\infty))$.  

Since $\rho \in BV([0,+\infty))$ we conclude that $\rho(t)$ admits a limit $\rho^{\infty}$ as $t \to \infty$. The fact that $\rho^{\infty} = r_M$ can be proved via contradiction. Suppose that $\rho^{\infty} < r_M$ and consider $\e > 0$ such that $r(x) > r_M - \e$ for all $x \in B(x_i,\e)$, where $B(x_i,\e)$ is the ball of centre $\bar{x}_i\in \arg \max(r)$ and radius $\e$. Since $\rho(t) \to \rho^{\infty}$ as $t \to \infty$, if $\rho^{\infty} < r_M$ then for $\e$ small enough there exists $\tau_{\e} > 0$ such that $\rho(t) <  r_M - 2 \e$ for all $t \geq \tau_{\e}$. Solving the IDE~\eqref{IDE} complemented with~\eqref{defR} for $t \geq \tau_{\e}$ gives
\beq
\label{Implicitteps} 
n(t,x) = n(\tau_{\e},x)  \, \displaystyle{e^{r(x) (t - \tau_{\e}) - \int_{\tau_{\e}}^t \rho(s) \, \dd s }}.
\eeq
Integrating both sides of~\eqref{Implicitteps} over $\Omega$ and estimating from below we find
$$
\rho(t) \; \geq \; \int_{B(x_i,\e)} n(\tau_{\e},x)  \, \displaystyle{e^{(r_M - \e) (t - \tau_{\e}) - \int_{\tau_{\e}}^t \rho(s) \, \dd s }} \, \ddx \; \geq \; \displaystyle{e^{\e (t - \tau_{\e})}} \; \int_{B(x_i,\e)} n(\tau_{\e},x) \, \ddx \quad \forall \; t \geq \tau_{\e},
$$
which implies that $\rho(t)  \to \infty$ as $t \to \infty$. Thus we arrive at a contradiction. Now suppose that $\rho^{\infty} > r_M$. If so, there exist $\e>0$ sufficiently small and $\tau_{\e} > 0$ sufficiently large so that $\rho(t) >  r_M + \e$ for all $t \geq \tau_{\e}$. Solving the IDE~\eqref{IDE} complemented with~\eqref{defR} for $t \geq \tau_{\e}$ gives~\eqref{Implicitteps}. Moreover, integrating both sides of~\eqref{Implicitteps} over $\Omega$ and estimating from above yields
$$
\rho(t) \; \leq \; \int_{\Omega} n(t_{\e},x)  \, \displaystyle{e^{r_M (t - \tau_{\e}) - \int_{\tau_{\e}}^t \rho(s) \, \dd s }} \, \ddx \; \leq \; \displaystyle{e^{- \e (t - \tau_{\e})}} \; \int_{\Omega} n(\tau_{\e},x) \, \ddx \quad \forall \; t \geq \tau_{\e},
$$
which implies that $\rho(t)  \to 0$ as $t \to \infty$. Thus we arrive again at a contradiction. In so doing we have proved that $\rho^{\infty} = r_M$.

Since the sequence $(n(t, \cdot))_{t>0}$ is bounded in $L^1(\Omega)$, the Banach-Alaoglu Theorem ensures that it is relatively (weakly-$*$) compact in $\mathcal{M}^1(\overline\Omega)$. Thus we can extract a subsequence $n(t_k,\cdot) \in \mathcal{M}^1(\overline\Omega)$ such that
$$
n(t_k,\cdot) \xrightharpoonup[k  \rightarrow \infty]{} n^{\infty} \quad \text{with} \quad n^{\infty} \in \mathcal{M}^1(\overline\Omega),
$$
where the measure $n^{\infty}$ is non-negative and its total mass is $r_M$. Finally, since solving the Cauchy problem~\eqref{Basis} yields
$$
n(t_k,x) = n^0(x)  \, \displaystyle{e^{r(x) t_k - \int_{0}^{t_k} \rho(s) \, \dd s }}
$$
and $\rho(t) \to r_M$ as $t \to \infty$, we have
$$
\int_{\Omega} \varphi(x) \, n(t_k, x) \, \ddx  \xrightarrow[k\rightarrow\infty] \, 0 \quad \forall \,  \varphi \in C\big(\overline{\Omega}\big) \; \text{ s.t. } \supp(\varphi) \cap \arg \max(r) = \emptyset,
$$
which implies that
$$
n^{\infty} = r_M \, \sum_{i=1}^N a_i \, \delta(x-x_i) \quad \text{with} \quad \sum_{i=1}^N a_i =1.
$$
This concludes the proof of Theorem~\ref{Theorem1} in the case where assumption~\eqref{rm} is satisfied.

{\bibliography{ArticleSelection}
\bibliographystyle{acm}}

         \end{document}